\documentclass{article}

\usepackage{array}
\newcolumntype{T}{>{\ttfamily}l}

\usepackage{graphicx}
\usepackage{amsmath}
\usepackage{amsthm}
\usepackage{amsfonts}
\usepackage{amssymb}
\usepackage{mathrsfs}
\usepackage{hyperref}

\usepackage[table,dvipsnames]{xcolor}

\usepackage[percent]{overpic}

\newtheorem{theorem}{Theorem}[section]

\newtheorem{lemma}[theorem]{Lemma}

\newtheorem{proposition}[theorem]{Proposition}

\theoremstyle{definition}
\newtheorem{definition}[theorem]{Definition}

\newtheorem*{theorema*}{Theorem A}
\newtheorem*{corollaryb*}{Corollary B}

\newcommand{\R}{\mathbb{R}}

\newcommand{\dvec}[1]{\overrightarrow{#1}}

\usepackage{subcaption}
\usepackage{tikz}

\begin{document}

\title{An immersed flat polyhedral Klein bottle}
\author{Stepan Paul}

\maketitle

\begin{abstract}
We present a polyhedral surface in Euclidean 3-space with the topology of a Klein bottle such that every vertex has zero angle defect and the star of every vertex is embedded. From the perspective of metric geometry, the polyhedron can be viewed as the image of a piecewise smooth isometric immersion of a flat Klein bottle. It is apparently the first such explicit example.
\end{abstract}

In general, we say a polyhedral surface in $\R^3$ is \emph{flat} if every vertex has zero angle defect and \emph{immersed} when the star of any vertex does not self-intersect. A discrete version of the Gauss-Bonnet Theorem tells us that if a compact polyhedral surface without boundary is flat, it must be a topological torus or Klein bottle. While flat tori have received much recent attention (see Section \ref{sec:background} for details), we focus here on the case of the Klein bottle.

\begin{theorema*}
 The polyhedral surface $\nabla$ constructed in Section \ref{sec:constructions} is flat and immersed in $\R^3$ and has the topology of a Klein bottle.
\end{theorema*}

\begin{figure}
 \centering
 \includegraphics[width=\textwidth]{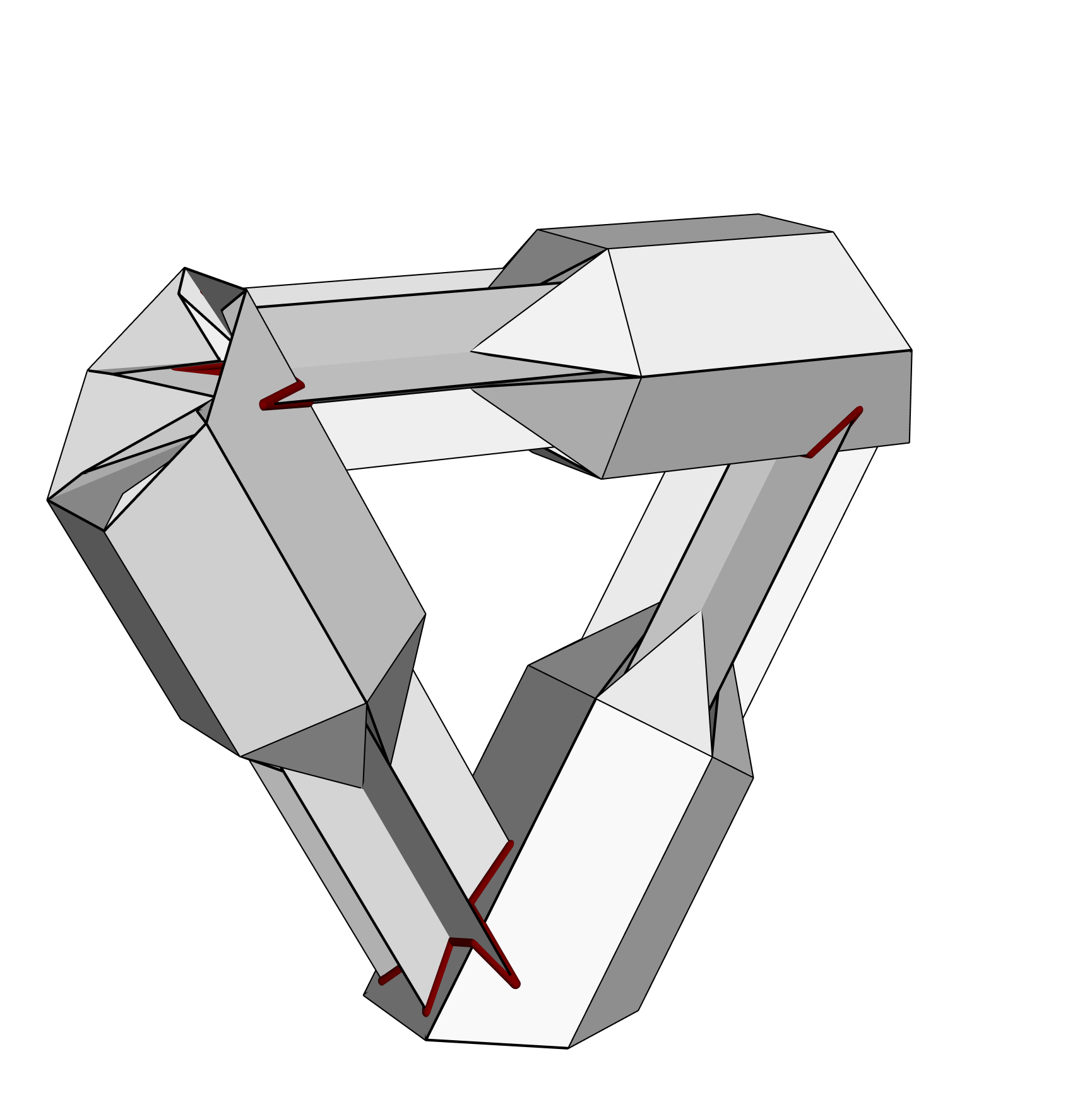}
 \caption{The polyhedral surface $\nabla$. The self-intersection curves are shown in maroon. We can understand the image as consisting of tubes with regular hexagonal and Star-of-David shaped cross-sections connected to each other by complicated sets of triangular faces---one type of connection along the edges of the triangular structure, and another type that folds the star-shaped tube up into the hexagonal tube at an angle so that it crashes through the side of the wider hexagonal tube.}
 \label{fig:flat-kb}
\end{figure}

The example $\nabla$, which has $108$ vertices and $168$ faces (after merging coplanar faces), is apparently the first explicit example of an immersed flat polyhedral Klein bottle. While existence of a polyhedral surface with these properties is implied by the results of Burago and Zalgaller \cite{bz1995}, the proofs operate at a high level of generality and involve tens of thousands of vertices. The contribution of Theorem A is that it exhibits a specific and constructive example that is comprehensible and generalizable, and which may serve as a starting point for lines of inquiry that have been fleshed out for flat tori but remain open for flat Klein bottles.

 From the perspective of metric geometry, the example can be thought of as an isometric immersion---defined as a locally injective path isometry---from a flat Klein bottle into $\R^3$; again it seems to be the first such explicit example although existence is implied by the work of Nash \cite{nash54} and Kuiper \cite{kuiper55}. Since a Klein bottle cannot be embedded into $\R^3$, an isometric immersion preserves about as much of the topology and geometry of a flat Klein bottle as one could hope for.
 
\begin{corollaryb*}
 The realization map for the polyhedral surface $\nabla$ constructed in Section \ref{sec:constructions} is an isometric immersion of a flat Klein bottle into $\R^3$ with respect to its induced length structure.
\end{corollaryb*}

Figure \ref{fig:flat-kb} depicts the polyhedral surface $\nabla$. The accompanying datasets include the combinatorial structure---given as a list of ordered tuples in $\{1,2,\ldots,108\}$ defining the faces---and a list of numerical approximations to the points in $\R^3$ to which the vertices map.\footnote{The data set for a similarly constructed surface using different parameters was first presented by the author in \cite{paul2025piecewise}.} We can understand the flatness property intuitively by saying that the polyhedron can (locally) be folded from flat paper with cut-outs allowing for the self-intersections.\footnote{In fact, an origami sculpture of such a polyhedron constructed by the author appeared in the Maison Poincar\'e in Paris \cite{paul2026flat}.}

The polyhedral surface is constructed in three stages.
\begin{enumerate}
 \item We start by giving a construction for polyhedral rectangular strips we call ``zee-bridges'' in Section \ref{sec:zee-bridge}.
 \item Next we show how to glue several zee-bridges together to form topological annuli that we call ``tube connections'' in Section \ref{sec:tube-joint}. These tube connections can be thought of as polyhedral transitions from a cylinder over one polygon to a cylinder over another.
 \item In Section \ref{sec:constructions}, we show how to glue six tube connections together to form the polyhedral surface $\nabla$ shown in Figure \ref{fig:flat-kb}.
\end{enumerate}

Along the way, we prove that the satisfaction of a certain set of inequalities ensures flatness at each vertex in Theorem \ref{thm:ac-lemma} and Lemma \ref{lem:rectangluar-zee-bridge-to-flat-tubes}. And in Section \ref{sec:local-injectivity}, we describe sufficient inequalities for ensuring the immersion property.

 Computational verification of these inequalities for the main construction is the content of Theorems \ref{thm:example-flat} and \ref{thm:example-locally-injective}. In fact, an open set in the parameter space leads to a family of such flat immersed polyhedral Klein bottles, and many variations on the construction presented here also satisfy these conditions. A thorough exploration of the parameter space is the subject of future study.




Before getting started on the details, we give some preliminary definitions and background in Sections \ref{sec:preliminaries} and \ref{sec:background} respectively.


%
%

\section{Preliminaries} \label{sec:preliminaries}

Here we frame the main construction in terms of the vocabulary of polyhedral and metric geometry. For more details, see, for example, \cite{cromwell1997polyhedra} for notions in polyhedral geometry and \cite{burago2001course} for background on metric geometry.

\subsection{Polyhedral surfaces}

Formally, for us, a \emph{polyhedral surface} (possibly with boundary) in $\R^3$ consists of a combinatorial structure together with a vertex realization map. The \emph{combinatorial structure} $S=(V,F)$ of a polyhedral surface consists of a finite set $V$---the \emph{vertices}---and a collection $F$---the \emph{faces}---of cyclically ordered subsets of $V$ so that each pair of distinct vertices appears as a consecutive pair---called an \emph{edge}---in exactly two faces (in the case of an \emph{interior edge}), one face (in the case of a \emph{boundary edge}), or zero faces. For cyclically ordered tuples, we write $[v_0,\ldots,v_{k-1}]$ with the square brackets indicating that indices are taken modulo $k$.

 The combinatorial structure $S$ can be identified with an abstract regular CW complex with the topology of a 2-manifold with boundary---what we'll call a \emph{CW surface}---and we often represent the combinatorial structure with a planar diagram of this CW surface with edge identifications indicated as needed.

 A \emph{vertex realization map} for $S$ is a mapping $\phi^0:V\rightarrow\R^3$ with the property that each face maps to the cyclically ordered vertices of a strictly convex planar polygon, which we call the \emph{image} of that face; the \emph{image} of the entire polyhedral surface is the union of the images of the faces. The \emph{image} of an edge between two points is the line segment connecting their images.
 
 We say that a polyhedral surface is \emph{flat} at an interior vertex $u$ when the angle sum at $u$ is exactly $2\pi$ and \emph{flat} overall if it is at every vertex.
 
 
 A polyhedral surface is said to be \emph{embedded} if the images of two faces intersect only at the image of a shared vertex or a shared edge if they have one.
  We say a polyhedral surface is \emph{immersed} if the star of each vertex---the polyhedral surface comprised of the faces containing it and their vertices---is embedded.
 
  When we talk about a geometric operation (translation, reflection, etc) being performed on a polyhedral surface $(S,\phi^0)$, it will be understood that it is to be composed with $\phi^0$. And when we talk about the topology of $(S,\phi^0)$, we are talking about that of $S$ as a CW surface, noting that this may be different from the topology of the image if it is not embedded.
  
 \subsection{Piecewise Linear Maps}
   
  If we triangulate the abstract CW surface $S$ associated with the combinatorial structure of a polyhedral surface without adding any vertices, we can define a piecewise linear (PL) map $\phi$ from the resulting simplicial complex (with respect to the barycentric coordinates) to $\R^3$ by extending the vertex realization map $\phi^0$, and thus obtain a map $\phi:S\rightarrow\R^3$, which we call the \emph{extended realization map} (or just the \emph{realization map}), that embeds each face. The image of $\phi$ is then the same as the image of the polyhedral surface as defined above. The map $\phi$ being a topological immersion\footnote{A topological immersion is a continuous locally injective map (see Chapter 4 of \cite{lee2000introduction}, for example) which is different from a smooth immersion; there is no assumption of differentiability.} is equivalent to the polyhedral surface being immersed as defined above.
 
\subsection{Metric Geometry}

 From the perspective of metric geometry, $\phi$ induces a pullback length structure on each face, and these can be patched together with respect to the combinatorial structure to give a global length structure on $S$. Note that this construction allows the image of $\phi$ to self-intersect without violating positive-definiteness of the length structure on $S$ since distances are induced one face at a time, and gluing happens with respect to the combinatorial structure.

 Under this induced length structure on $S$, $\phi$ is automatically a path isometry. In this context, an injective path isometry is called an \emph{isometric embedding} and a locally injective path isometry is called an \emph{isometric immersion}, coinciding with the notions of embedded and immersed polyhedral surfaces above.
 
 Flatness in metric geometry is defined as being locally isometric to a Euclidean space. Under the induced length structure of a PL map, $S$ is automatically flat on the interior of each face and edge, and flatness at the vertices is equivalent to the angle sum version of the flatness condition on polyhedral surfaces above. Thus from the perspective of metric geometry, the main construction gives us an isometric immersion of a flat Klein bottle into $\R^3$.
 
 Relating this to differential geometry, if a manifold $M$ with a compatible length structure is flat, then it is isometric to a smooth Riemannian manifold. Applying this to surfaces, by the Gauss-Bonnet Theorem, any compact surface without boundary with a compatible globally flat length structure must have Euler characteristic zero and is thus a topological torus or Klein bottle.

\section{Background} \label{sec:background}

As mentioned in the introduction, the story of the flat Klein bottle mostly mirrors that of the flat torus, but remains largely untold. While neither a $\mathcal C^2$ isometric embedding of a flat torus nor an immersion of a flat Klein bottle exists, the existence of $\mathcal C^1$ maps with these properties is implied by the famous results of Nash \cite{nash54} and Kuiper \cite{kuiper55} in the 1950s. Burago and Zalgaller proved results in 1960 that imply the existence of a PL isometric embedding of a flat torus \cite{bz1960}, and in 1995 of a PL isometric immersion of a flat Klein bottle \cite{bz1995}. The first explicit examples of isometrically embedded flat tori seem to be a family of polyhedral examples due to Brehm in 1978 \cite{brehm1978}. And in 2012, Borrelli, Jabrane, Lazarus, and Thibert gave an explicit $\mathcal C^1$ embedding of a flat torus \cite{bjlt12}.

Polyhedral flat tori have attracted recent attention with proofs that every flat torus can be embedded under Brehm's construction \cite{tsuboi2024} and that the minimum number of vertices for a flat polyhedral torus is eight \cite{schwartz2026}. Such questions have not been addressed at all for flat Klein bottles.
 
In 2021, the author exhibited a piecewise smooth path isometry from a flat Klein bottle into $\R^3$ using the theory of curved-crease origami \cite{me2021}; its image is shown in Figure \ref{fig:cc}. And in a 2025 preprint, Hisakawa, Kaji, and Kawai construct a flat polyhedral Klein bottle \cite{hkk2025polyhedra}. Both constructions fail to be locally injective/immersed at finitely many points.

\begin{figure}
 \begin{subfigure}[b]{0.45\textwidth}
    \centering
    \includegraphics[width=\textwidth]{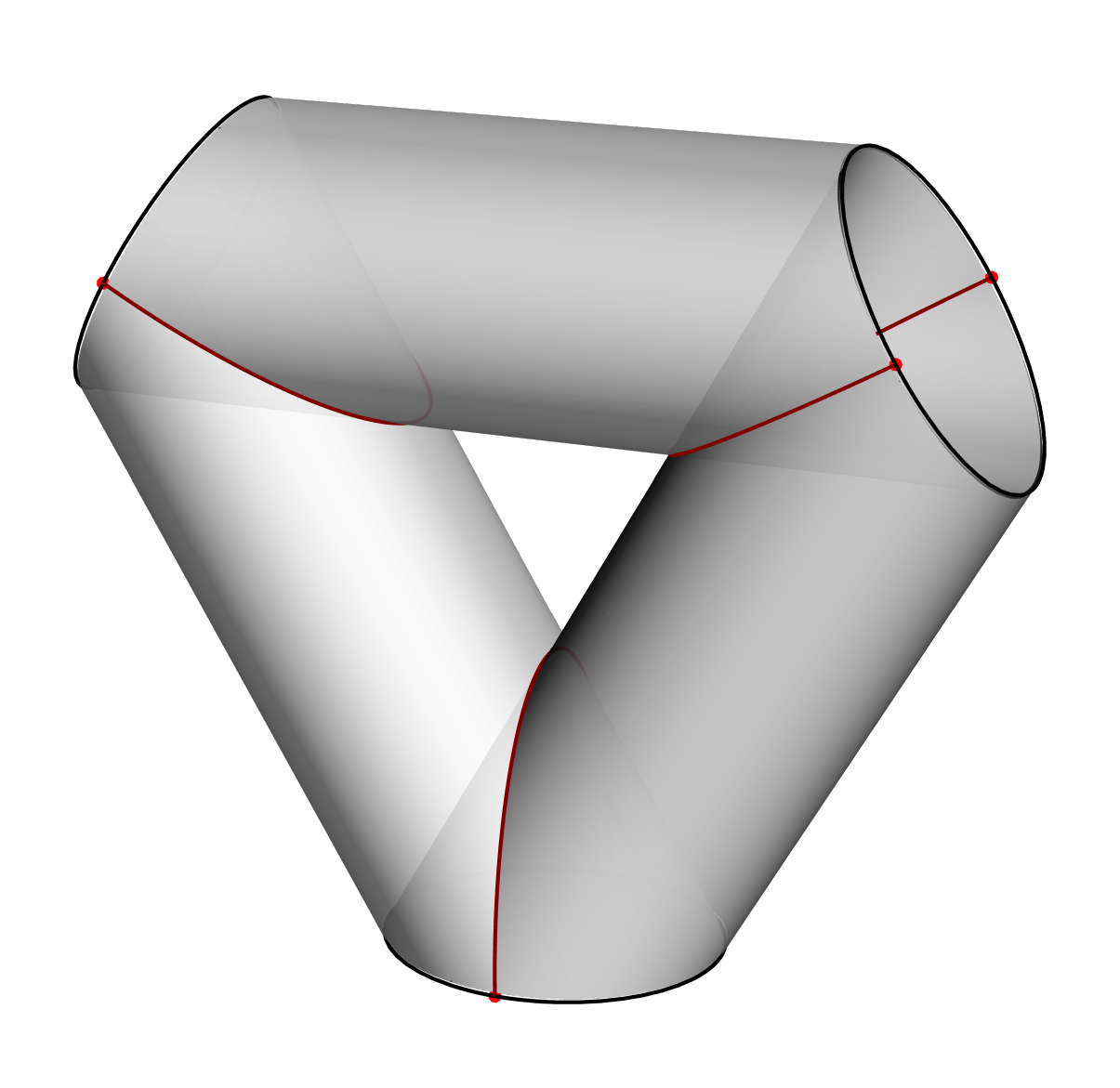}
    \caption{}
    \label{fig:cc}
  \end{subfigure}
  \hfill 
  \begin{subfigure}[b]{0.45\textwidth}
    \centering
    \includegraphics[width=\textwidth]{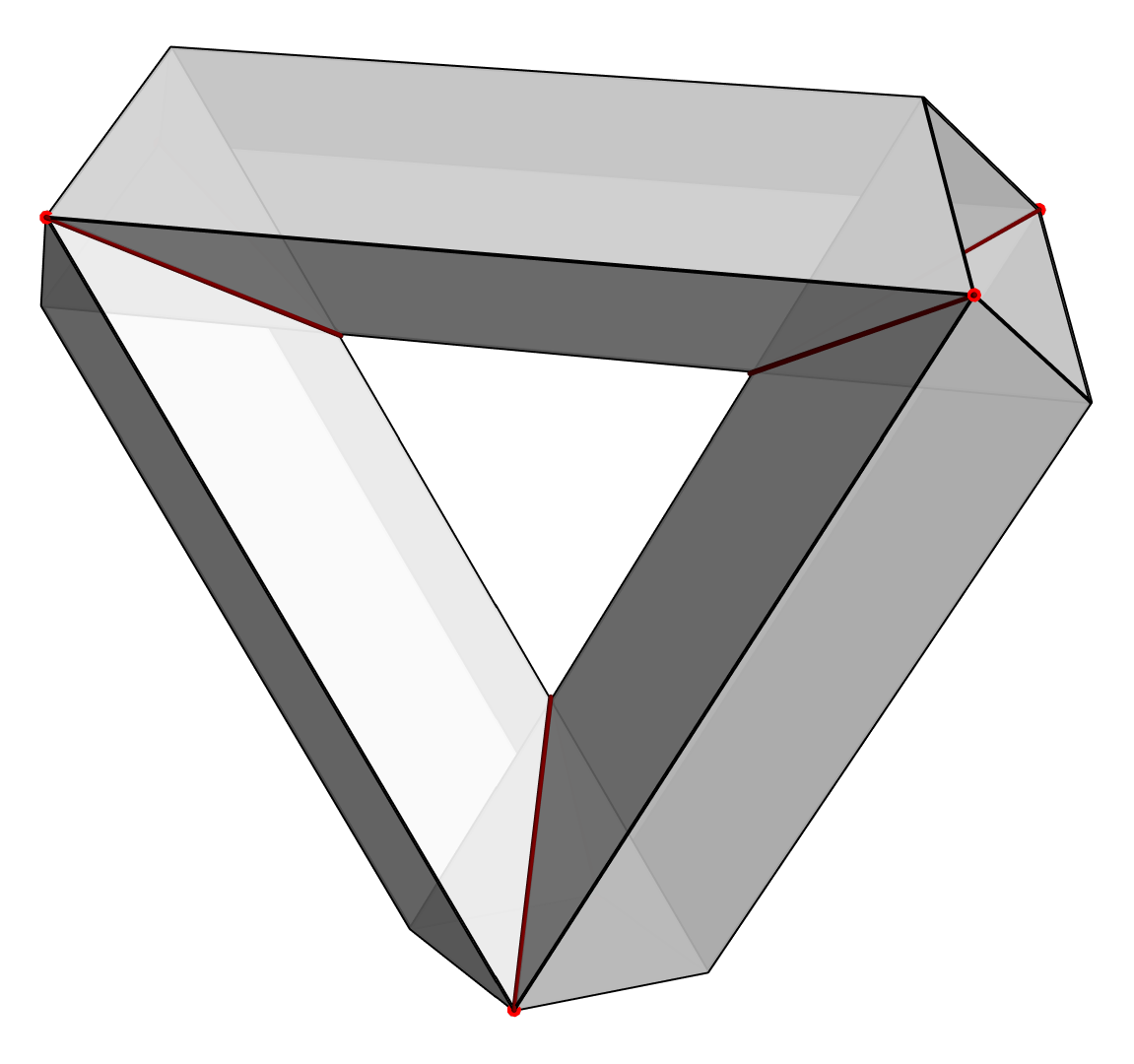}
    \caption{}
    \label{fig:square}
  \end{subfigure}
 
 \caption{(a) The image of the piecewise smooth path isometry of a flat Klein bottle from \cite{me2021}. The three full ellipses can be thought of as curved-crease origami folds, and the three maroon half-ellipses are the self intersections. The map fails to be locally injective at the pre-images of six red points where the ellipses and half-ellipses intersect. (b) A polyhedral version of the previous using square-shaped tubes.} \label{fig:non-immersed}
\end{figure}

Thus the primary innovation of the present article has to do with the immersion property. The intuition behind the main construction is that we can start by discretizing the piecewise cylindrical construction from \cite{me2021}. Na\"ively switching from circular cylinders to convex polygonal ones as in Figure \ref{fig:square} can be done in a way that preserves flatness but the results will necessarily still have non-immersed vertices. However, we can use the idea behind the embedded flat torus construction of Quintanar \cite{quintanar2020} to ``shrink'' a convex polygonal cylinder into a star-shaped one of smaller diameter via a sequence of triangular faces connecting the two. The technical parts of the construction we present deal with generalizing the construction to allow the convex polygonal cylinder and star-shaped cylinder to be centered along different axes, and confirming that this shrinking operation allows the self-intersections to happen away from the vertices so that the vertices are immersed.

\section{Zee-Bridges}\label{sec:zee-bridge}

The fundamental building block of the constructions we present is the ``rectangular zee-bridge'' (so named because of the geometry seen in Figure \ref{fig:s-geometric}), which we define below.

Let $W,X,Y,Z\in\R^3$ be points such that $\hat u=\dvec{WX}$ and $\hat v=\dvec{YZ}$ are unit vectors, and let $\hat s,\hat t$ be unit vectors with 
$$\hat s\cdot\hat u=0 \qquad \hat t\cdot\hat v=0.$$
In this case, we call $\mathcal F=(W,X,Y,Z,\hat s,\hat t)$ a \emph{bridge frame}.

A \emph{zee-bridge} $\mathcal Z=(S,\phi^0)$ for the bridge frame $\mathcal F$ is a polyhedral surface with boundary with the combinatorial structure $S$ shown in Figure \ref{fig:s} with vertex realization map $\phi^0$ defined by
$$\phi^0(w)=W \qquad \phi^0(x)=X \qquad \phi^0(y)=Y \qquad \phi^0(z)=Z$$
\begin{equation} \label{eq:abcd}
\begin{array}{ll}A:=\phi^0(a)=W+\alpha\hat s & B:=\phi^0(b)=X+\beta\hat s\\
C:=\phi^0(c)=Y-\gamma\hat t & D:=\phi^0(d)=Z-\delta\hat t
\end{array}
\end{equation}
for some $\alpha,\beta,\gamma,\delta>0$, which we call the \emph{parameters} of the zee-bridge. Note that under this construction, the images of the faces $[w,x,a,b]$ and $[z,y,c,d]$ are strictly convex planar trapezoids with right angles at the vertices $W,X,Y,Z$.

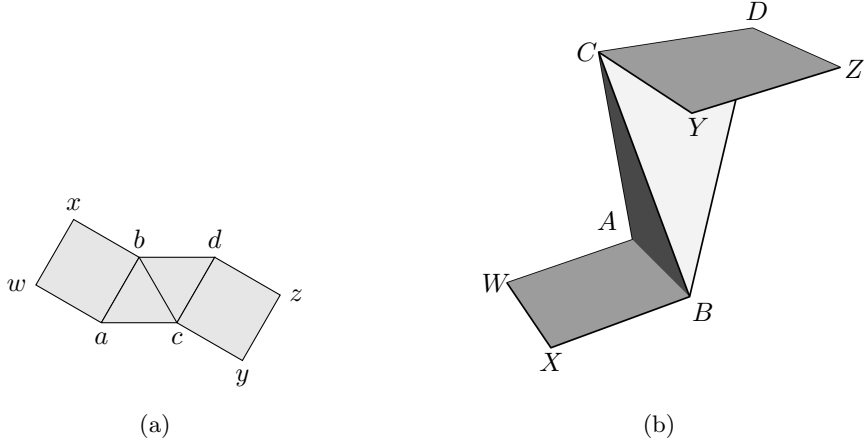
\begin{figure}
\centering
 \begin{subfigure}[b]{0.45\textwidth}
    \centering
    \begin{tikzpicture}
 
 \draw[fill=gray!20] (0,0) node[anchor=north] {$a$} -- (60:1) node[anchor=south] {$b$} -- (105:{sqrt(2)}) node[anchor=south] {$x$} -- (150:1) node[anchor=east]{$w$} -- (0,0);
 \draw[fill=gray!20] (0,0) -- (1,0) node[anchor=north] {$c$} -- +(60:1) node[anchor=south] {$d$} -- +(120:1) -- (0,0);
 
 \draw[fill=gray!20] (1,0) -- +(-30:1) node [anchor=north] {$y$} -- +(15:{sqrt(2)}) node[anchor=west] {$z$} -- +(60:1) -- (1,0);
 
 \draw (1,0) -- (60:1);

\end{tikzpicture}
    \caption{}
    \label{fig:s}
  \end{subfigure}
  \hfill
  \begin{subfigure}[b]{0.45\textwidth}
    \begin{overpic}[width=\textwidth,tics=10]{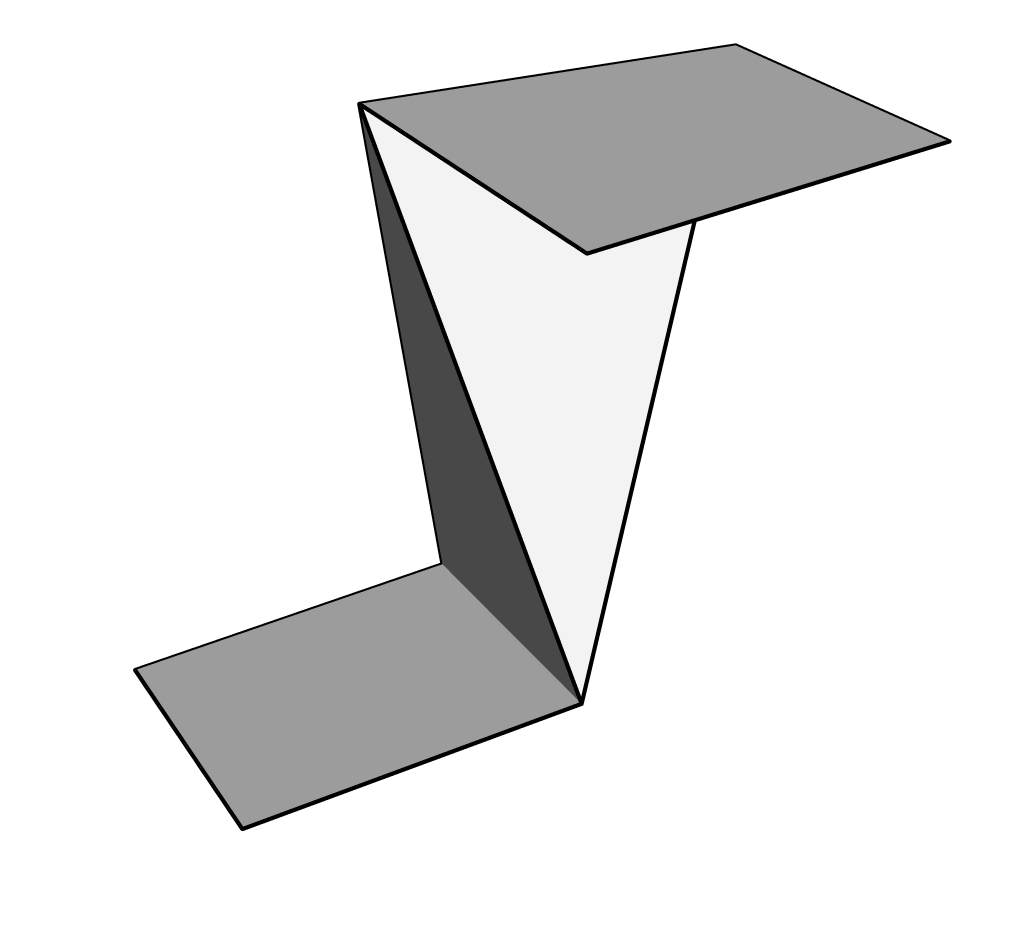}
 \put (21,6) {$X$}
 \put (7,25) {$W$}
 \put (35,40) {$A$}
 \put (58,18) {$B$}
 \put (30,81) {$C$}
 \put (71,91) {$D$}
 \put (57,63) {$Y$}
 \put (95,77) {$Z$}
\end{overpic}
    \caption{}
    \label{fig:z}
  \end{subfigure}
  \caption{(a) The combinatorial structure $S$ of a zee-bridge. (b) The image of a zee-bridge in $\R^3$.}

\end{figure}

For a bridge frame $\mathcal F=(W,X,Y,Z,\hat s,\hat t)$, define $\tilde{\mathcal F}$ to be the bridge frame $(Z,Y,X,W,-\hat t,-\hat s)$. A zee-bridge for $\mathcal F$ with parameters $\alpha,\beta,\gamma,\delta$ then has the same image as a zee-bridge for $\tilde{\mathcal F}$ with parameters $\delta,\gamma,\beta,\alpha$.

We say a zee-bridge is \emph{rectangular} if the angle sums at $a$, $b$, $c$, and $d$ are each $\pi$. In this case, the induced length structure on $S$ is that of a planar rectangle as shown in Figure \ref{fig:s-geometric}.

 \begin{figure}
  \centering
  \begin{tikzpicture}[scale=1]
   \def\w{9}
   \def\wa{2.5}
   \def\wb{4}
   \def\wc{5}
   \def\we{7}
   
   \draw[fill=gray!20] (0,0) node[anchor=north east] {$w$} -- (\w,0) node[anchor=north west] {$y$} -- node[anchor=west] {$1$} (\w,1) node[anchor=south west] {$z$} -- (0,1) node[anchor=south east] {$x$} -- node[anchor=east] {$1$}(0,0);
   \draw (\wa,0) node[anchor=north] {$a$} -- (\wb,1) node[anchor=south] {$b$} -- (\wc,0) node[anchor=north] {$c$} -- (\we,1) node[anchor=south] {$d$};
   
   
   \draw (\wa/2,0) node[anchor = north] {$\alpha$};
   \draw ({(\wc+\w)/2},0) node[anchor = north] {$\gamma$};
   \draw (\wb/2,1) node[anchor = south] {$\beta$};
   \draw ({(\we+\w)/2},1) node[anchor = south] {$\delta$};
   
   \draw ({(\wa+\wc)/2},0) node[anchor=north] {$|AC|$};
   \draw ({(\wb+\we)/2},1) node[anchor=south] {$|BD|$};
   
  \end{tikzpicture}
  \caption{The combinatorial structure $S$ with the length structure induced by a rectangular zee-bridge. The lengths of the boundary edges are labelled.}\label{fig:s-geometric}
 \end{figure}
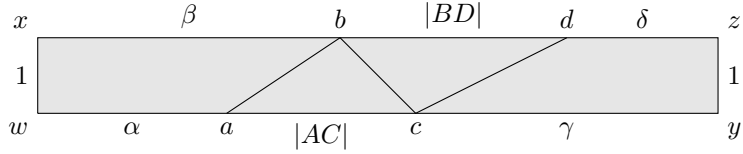

\begin{theorem}\label{prop:zee-bridge}
 Let $\mathcal F=(W,X,Y,Z,\hat s,\hat t)$ be a bridge frame, and let $\mathcal Z$ be a zee-bridge with parameters $\alpha,\beta,\gamma,\delta$. Then in the notation above, $\mathcal Z$ is rectangular if and only if
   \begin{equation}
    \alpha+|AC|+\gamma=\beta+|BD|+\delta \quad \text{and} \quad |BC|=\sqrt{1+(|AC|+\alpha-\beta)^2}. \label{eq:zee-bridge}
   \end{equation} 
\end{theorem}

\begin{proof}
 The forwards direction is clear from Figure \ref{fig:s-geometric} since faces of the image are congruent to faces of the domain under the induced length structure. For the backwards direction, we will show that if the equations of (\ref{eq:zee-bridge}) are satisfied, then the corresponding congruence of faces still holds so that the angle sums in the image will be as in Figure \ref{fig:s-geometric}.
 
 First we note that by construction, the trapezoids $[W,X,B,A]$ and $[Z,Y,C,D]$ are congruent to the trapezoids $[w,x,b,a]$ and $[z,y,c,d]$ in Figure \ref{fig:s-geometric}. Then we note that the triangles $[A,B,C]$ and $[a,b,c]$ are congruent since their respective edge lengths are equal---$|AB|=|ab|$ by the trapezoid congruence, $|AC|=|ac|$ by the induced length structure, and $|BC|=|bc|$ by the second equation of (\ref{eq:zee-bridge}). The equations together imply that $|BC|=\sqrt{1+(|BD|+\delta-\gamma)^2}$, so we can similarly prove triangles $[D,C,B]$ and $[d,c,b]$ are congruent.
\end{proof}

Later on, we will need to find $\beta$ and $\delta$ which make a zee-bridge rectangular (if they exist) given $\alpha$ and $\gamma$. The next two lemmas give candidate values for $\beta$ and $\delta$, and Theorem \ref{thm:ac-lemma} gives a sufficient condition for their validity.

\begin{lemma}\label{lem:b}
 Given a bridge frame $\mathcal F$ and $\alpha,\gamma>0$, in the notation above, let $C'=A+|AC|\hat s$. Then there exists $\beta\in\R$ such that $B=X+\beta\hat s$ is equidistant from $C$ and $C'$. Furthermore, if $C\neq C'$,
 \begin{equation}\label{eq:beta}
  \beta=\alpha+\frac{\dvec{AC}\cdot\hat u}{|AC|-\dvec{AC}\cdot\hat s}
 \end{equation}
 is the unique value with this property.
\end{lemma}

\begin{proof}
 If $C$ lies along the closed ray from $A$ in direction $\hat s$, then $C=C'$ so the conclusion is trivial.
 
 Otherwise, $C\neq C'$, and the plane of points equidistant from $C$ and $C'$ has normal vector $\dvec{CC'}=|AC|\hat s-\dvec{AC}$ and passes through $A$. Then $B=A+\hat u+(\beta-\alpha)\hat s$ lies on this plane if and only if
 $$\dvec{CC'}\cdot\dvec{AB}=0$$
 or equivalently
 $$(|AC|\hat s-\dvec{AC})\cdot(\hat u+(\beta-\alpha)\hat s)=0,$$
 which simplifies to
 $$\dvec{AC}\cdot\hat u=(\beta-\alpha)(|AC|-\dvec{AC}\cdot\hat s).$$
 Since $C$ does not lie along the closed ray from $A$ in direction $\hat s$, $|AC|\neq\dvec{AC}\cdot\hat s$, so solving for $\beta$ shows that (\ref{eq:beta}) is the unique solution as claimed.
\end{proof}

\begin{lemma}\label{lem:d}
 Given a bridge frame $\mathcal F$ and $\alpha,\gamma>0$, let $\beta$ be as in Lemma \ref{lem:b}. Then set $B'=C+\hat v-\ell\hat t$ where $\ell=\alpha-\beta+|AC|$. If $\dvec{BB'}\cdot\hat t\neq0$, then there is a unique value of $\delta$ such that $D=Z-\delta\hat t$ is equidistant from $B$ and $B'$ given by
 \begin{equation}\label{eq:delta}
  \delta=\gamma+\frac{\dvec{BB'}\cdot\hat v}{\dvec{BB'}\cdot\hat t}=\gamma+\frac{\dvec{BC}\cdot\hat v+1}{\dvec{BC}\cdot\hat t-\ell}.
 \end{equation}
\end{lemma}

%

\begin{proof}
 The points $B$ and $B'$ are equidistant from $C$ by construction, so the plane of points equidistant from $B$ and $B'$ contains $C$ and has normal vector $\dvec{BB'}$. Then $D=C+\hat v+(\gamma-\delta)\hat t$ lies on this plane if and only if
 $$\dvec{BB'}\cdot\dvec{CD}=0$$
 which simplifies to
 $$\dvec{BB'}\cdot(\hat v+(\gamma-\delta)\hat t)=0.$$
 Equivalently,
 $$\dvec{BB'}\cdot\hat v=(\dvec{BB'}\cdot\hat t)(\delta-\gamma).$$
 Solving for $\delta$ shows that (\ref{eq:delta}) is the unique solution; recognizing that $\dvec{BB'}=\dvec{BC}+\hat v-\ell\hat t$ and $\hat t\cdot\hat v=0$ gives the second form of the solution.
\end{proof}

Mirroring (\ref{eq:beta}) and (\ref{eq:delta}), for a bridge frame $\mathcal F=(W,X,Y,Z,\hat s,\hat t)$, we define
\begin{equation}\label{eq:bars}
 \begin{array}{rll}
 \bar\beta(\mathcal F,\alpha,\gamma)&= \alpha+\frac{\dvec{AC}\cdot \hat u}{|AC|-\dvec{AC}\cdot\hat s},& \quad (|AC|\neq\dvec{AC}\cdot\hat s)\\
 \bar\delta(\mathcal F,\alpha,\gamma)&= \gamma+\frac{\dvec{BC}\cdot\hat v+1}{\dvec{BC}\cdot\hat t-\ell},& \quad (\dvec{BC}\cdot\hat t\neq\ell)
 \end{array}
\end{equation}
where, as before, $A=W+\alpha\hat s$, $C=Y-\gamma\hat t$, $B=X+\bar\beta(\mathcal F,\alpha,\gamma)\hat s$, and $\ell=\alpha+|AC|-\bar\beta(\mathcal F,\alpha,\gamma)$.


If $\alpha,\gamma>0$ and $\beta=\bar\beta(\mathcal F,\alpha,\gamma)$ and $\delta=\bar\delta(\mathcal F,\alpha,\gamma)$ are defined and positive, then we say $\alpha,\gamma$ \emph{generate} positive parameters for $\mathcal F$, and we write $\mathcal Z_{\mathcal F,\alpha,\gamma}$ for the zee-bridge with parameters $\alpha,\beta,\gamma,\delta$. In this case, we also define the \emph{discriminant} of $\mathcal Z_{\mathcal F,\alpha,\gamma}$ to be
 \begin{equation}\label{eq:discriminant}
  \Delta_{\mathcal F,\alpha,\gamma}=\alpha-\beta+|AC|+\gamma-\delta.
 \end{equation}

\begin{theorem}\label{thm:ac-lemma}
 Suppose $\mathcal F$ is a bridge frame and $\alpha,\gamma$ generate positive parameters for $\mathcal F$. If $\Delta_{\mathcal F,\alpha,\gamma}>0$, then $\mathcal Z_{\mathcal F,\alpha,\gamma}$ is rectangular.
\end{theorem}

\begin{proof}
 We start by noting that the second equation of (\ref{eq:zee-bridge}) is satisfied by Lemma \ref{lem:b}. And in the notation of Lemma \ref{lem:d}, since we have $\dvec{B'D}=(\gamma-\delta+\ell)\hat t$, 
 $$|BD|=|B'D|=|\gamma-\delta+(\alpha-\beta+|AC|)|=|\Delta_{\mathcal F,\alpha,\gamma}|.$$
 So if $\Delta_{\mathcal F,\alpha,\gamma}>0$, $|BD|=\Delta_{\mathcal F,\alpha,\gamma}$, so the first equation of (\ref{eq:zee-bridge}) is satisfied as well. Thus by Theorem \ref{prop:zee-bridge}, $\mathcal Z_{\mathcal F,\alpha,\gamma}$ is rectangular.
\end{proof}

To offer a geometric interpretation of this proof, note that if $\dvec{B'D}$ points in direction $\hat t$, then the discriminant is positive.

A special case that we'll use later on is covered below.

\begin{lemma}\label{lem:straight-implies-rect}
 Suppose $\mathcal F=(W,X,Y,Z,\hat s,\hat t)$ is a bridge frame with $\hat s=\hat t$ and that $\alpha,\gamma$ generate positive parameters for $\mathcal F$. Then $\mathcal Z_{\mathcal F,\alpha,\gamma}$ is rectangular.
\end{lemma}

\begin{proof}
 In the notation of Lemma \ref{lem:d}, we find that since $\dvec{B'C}\cdot\hat t=\ell=\dvec{BC'}\cdot\hat s$,
 \begin{align*}
  \dvec{B'D}\cdot\hat t - \dvec{BD}\cdot\hat t &=\dvec{B'C}\cdot\hat t - \dvec{BC}\cdot\hat t\\
  &=\dvec{BC'}\cdot\hat s - \dvec{BC}\cdot\hat s\\
  &=\dvec{AC'}\cdot\hat s - \dvec{AC}\cdot\hat s\\
  &=|AC|-\dvec{AC}\cdot\hat s>0.\\
 \end{align*}%
 Then $\dvec{B'D}\cdot\hat t > \dvec{BD}\cdot\hat t$, so since $\left|\dvec{B'D}\cdot\hat t\right|=|BD|>\left|\dvec{BD}\cdot\hat t\right|$, this implies $\dvec{B'D}$ points in the direction of $\hat t$, so the discriminant is positive.
\end{proof}


%
%

\section{Tube Connections} \label{sec:tube-joint}

 Our next step is to construct what we'll call ``tube connections''---polyhedral transitions from a cylinder over one polygon to a cylinder over another.
 
\subsection{Tube Frames and Tube Connections}

  For an embedded planar polygon $\mathcal P$ in $\R^3$ with cyclically numbered vertices $P_0,\ldots,P_{k-1}$, we will write $\mathcal P=[P_0,\ldots,P_{k-1}]$. If $\hat s$ is a normal vector to the plane containing $\mathcal P$, we say that the pair $(\mathcal P,\hat s)$ is \emph{right-handed} if the vertices go counterclockwise around the interior of $\mathcal P$ when viewed from a point in direction $\hat s$ from the plane containing $\mathcal P$; otherwise we say $(\mathcal P,\hat s)$ is \emph{left-handed}.
 
 
 Suppose that $\mathcal P=[P_0,\ldots,P_{2n-1}]$ and $\mathcal Q=[Q_0,\ldots,Q_{2n-1}]$ are equilateral $2n$-gons of side length 1, and let $\hat s$ and $\hat t$ be normal vectors to the planes containing $\mathcal P$ and $\mathcal Q$ respectively. In this case, we call $\mathscr F=(\mathcal P,\mathcal Q,\hat s,\hat t)$ a \emph{tube frame}.
 
 Let $T$ be the combinatorial structure shown in Figure \ref{fig:j} so that $T$ has the topology of an annulus. Note that each horizontal strip $S_i$ is homeomorphic to the combinatorial structure $S$ from Figure \ref{fig:s} via extension of the vertex mapping
 $$\begin{array}{llll}
 w_i\mapsto w & y_{i-1}\mapsto x & y_i\mapsto y & w_{i-1}\mapsto z\\
 a_i\mapsto a & c_{i-1}\mapsto b & c_i\mapsto c & a_{i-1}\mapsto d.
 \end{array}$$
 
 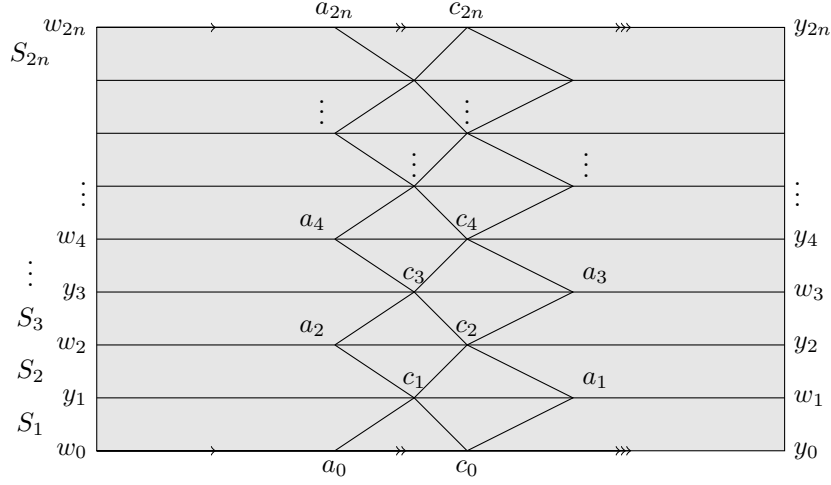
\begin{figure}
  \centering
  \begin{tikzpicture}[scale=0.7]
   \def\w{13}
   \def\h{8}
   \def\wa{4.5}
   \def\wb{6}
   \def\wc{7}
   \def\we{9}
   
   \draw[fill=gray!20] (0,0) -- (\w,0) -- (\w,\h) -- (0,\h) -- (0,0);
   \foreach \i in {1,...,\h}
    {
     \draw (0,\i) -- (\w,\i);
    }
    
    \foreach \i in {2,4,...,\h}
     {
      \draw (\wc,\i-2) -- (\we,\i-1) -- (\wc,\i);
      \draw (\wc,\i-2) -- (\wb,\i-1) -- (\wc,\i);
      \draw (\wa,\i-2) -- (\wb,\i-1) -- (\wa,\i);
     }
   
   \draw[->] (0,8) -- (\wa/2,8);
   \draw[->] (0,0) -- (\wa/2,0);
   
   \draw[->] (0,8) -- ({(\wa+\wc)/2},8);
   \draw[->] (0,0) -- ({(\wa+\wc)/2},0);
   \draw[->] (0,8) -- ({(\wa+\wc)/2+0.1},8);
   \draw[->] (0,0) -- ({(\wa+\wc)/2+0.1},0);
   
   \draw[->] (0,8) -- ({(\w+\wc)/2},8);
   \draw[->] (0,0) -- ({(\w+\wc)/2},0);
   \draw[->] (0,8) -- ({(\w+\wc)/2+0.1},8);
   \draw[->] (0,0) -- ({(\w+\wc)/2+0.1},0);
   \draw[->] (0,8) -- ({(\w+\wc)/2-0.1},8);
   \draw[->] (0,0) -- ({(\w+\wc)/2-0.1},0);
    
   \draw (0,0) node[anchor=east] {$w_0$};
   \draw (0,1) node[anchor=east] {$y_1$};
   \draw (0,2) node[anchor=east] {$w_2$};
   \draw (0,3) node[anchor=east] {$y_3$};
   \draw (0,4) node[anchor=east] {$w_4$};
   \draw (0,5) node[anchor=east] {$\vdots$};
   \draw (0,8) node[anchor=east] {$w_{2n}$};
   
   \draw (\w,0) node[anchor=west] {$y_0$};
   \draw (\w,1) node[anchor=west] {$w_1$};
   \draw (\w,2) node[anchor=west] {$y_2$};
   \draw (\w,3) node[anchor=west] {$w_3$};
   \draw (\w,4) node[anchor=west] {$y_4$};
   \draw (\w,5) node[anchor=west] {$\vdots$};
   \draw (\w,8) node[anchor=west] {$y_{2n}$};
   
   \draw (\we,1) node[anchor=south west] {$a_1$};
   \draw (\we,3) node[anchor=south west] {$a_3$};
   \draw (\we,5) node[anchor=south west] {$\vdots$};

   \draw (\wa,0) node[anchor=north] {$a_0$};
   \draw (\wa,2) node[anchor=south east] {$a_2$};
   \draw (\wa,4) node[anchor=south east] {$a_4$};
   \draw (\wa,6) node[anchor=south east] {$\vdots$};
   \draw (\wa,8) node[anchor=south] {$a_{2n}$};
   
   \draw (\wb,1) node[anchor=south] {$c_1$};
   \draw (\wb,3) node[anchor=south] {$c_3$};
   \draw (\wb,5) node[anchor=south] {$\vdots$};

   \draw (\wc,0) node[anchor=north] {$c_0$};
   \draw (\wc,2) node[anchor=south] {$c_2$};
   \draw (\wc,4) node[anchor=south] {$c_4$};
   \draw (\wc,6) node[anchor=south] {$\vdots$};
   \draw (\wc,8) node[anchor=south] {$c_{2n}$};
   
   \draw (-1.25,0.5) node {$S_1$};
   \draw (-1.25,1.5) node {$S_2$};
   \draw (-1.25,2.5) node {$S_3$};
   \draw (-1.25,3.5) node {$\vdots$};
   \draw (-1.25,7.5) node {$S_{2n}$};
   
  \end{tikzpicture}
  \caption{The combinatorial polyhedral structure $T$ with edge identifications indicated by arrowheads.}\label{fig:j}
 \end{figure}

 Define bridge frames $\mathcal F_i$ via
 \begin{equation}\label{eq:fi}
  \mathcal F_i=\begin{cases} (P_i,P_{i-1},Q_i,Q_{i-1},\hat s,\hat t)& \text{$i$ even} \\ (Q_i,Q_{i-1},P_i,P_{i-1},-\hat t,-\hat s)& \text{$i$ odd} \end{cases}.
 \end{equation}
 Note that the convention is that the even numbered frames go from $\mathcal P$ to $\mathcal Q$ (left to right in Figure \ref{fig:j}) and odd numbered bridge frames go from $\mathcal Q$ to $\mathcal P$ (right to left).
 
 A \emph{tube connection} for $\mathscr F$ is a polyhedral surface with boundary $\mathcal C=(T,\Phi^0)$ with combinatorial structure $T$ so that if $\phi^0_i$ is the restriction of $\Phi^0$ to the vertices of $S_i$, then the polyhedral surface $\mathcal Z_i=(S_i,\phi^0_i)$ is a zee-bridge for $\mathcal F_i$. In this way, the extended realization map $\Phi$ sends the left edge of $T$ in Figure \ref{fig:j} to $\mathcal P$ and the right edge to $\mathcal Q$. Letting $\alpha_i=|\Phi(w_i)\Phi(a_i)|$ and $\gamma_i=|\Phi(y_i)\Phi(c_i)|$, note that $(\alpha_{i-1},\gamma_{i},\gamma_{i-1},\alpha_{i})$ are the parameters for $\mathcal Z_i$.
 
 \begin{figure}
  \centering
  \includegraphics[width=3in]{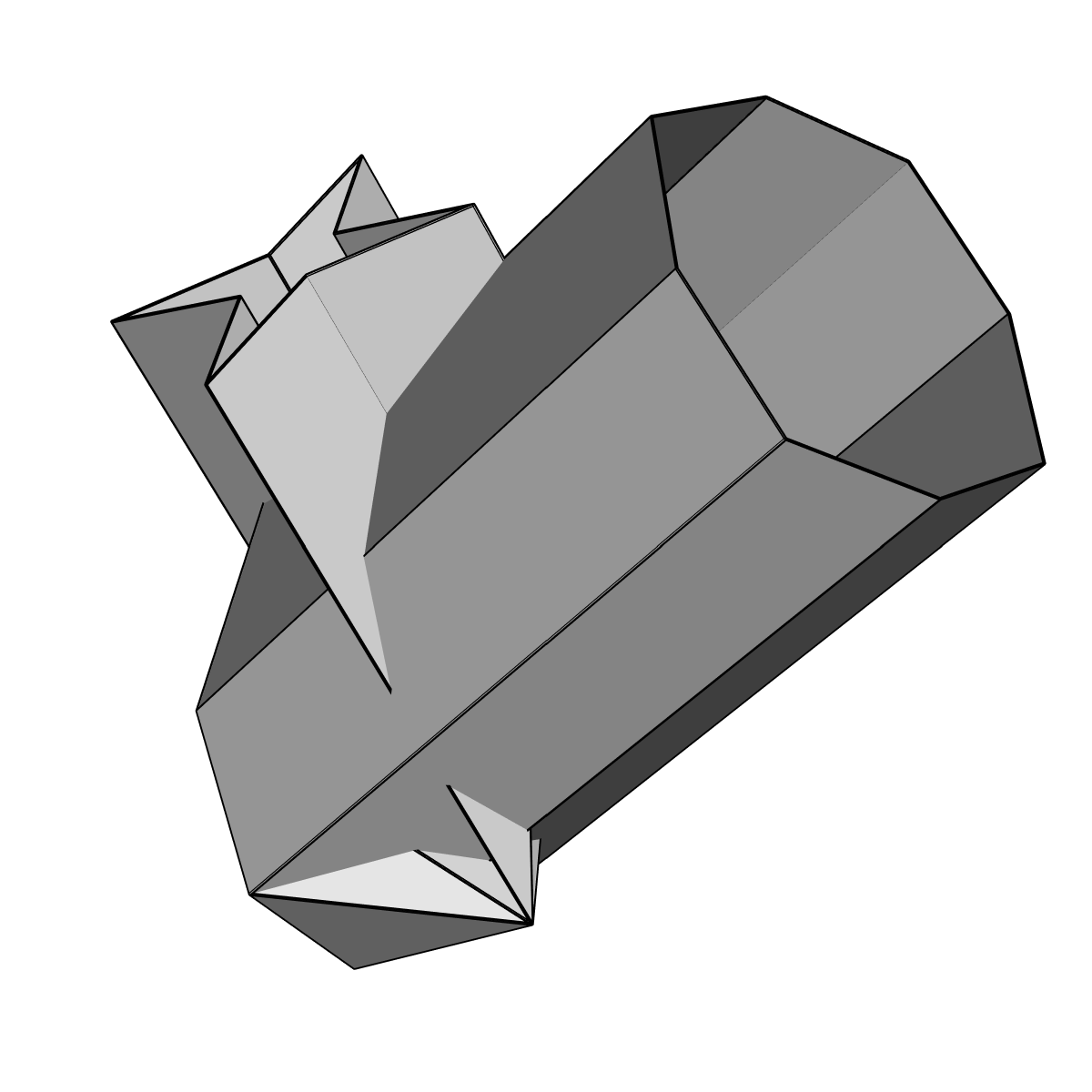}
  \caption{The image of a tube connection (with self-intersections). The two octagons $\mathcal P$ and $\mathcal Q$ are visible, but many of the vertices $\Phi(a_i),\Phi(c_i)$ are obscured.}
 \end{figure}
 
 For a given tube frame $\mathscr F$, ordered tuples of positive real numbers $\vec{\alpha}=[\alpha_0,\ldots,\alpha_{2n-1}]$ and $\vec{\gamma}=[\gamma_0,\ldots,\gamma_{2n-1}]$ determine a tube connection $\mathcal C_{\mathscr F,\vec{\alpha},\vec{\gamma}}$ for $\mathscr F$. In general, we call $\vec\alpha,\vec\gamma$ the \emph{parameters} for this tube connection.
 
 \begin{lemma}\label{lem:rectangluar-zee-bridge-to-flat-tubes}
  If $\mathcal C$ is a tube connection whose restriction $\mathcal Z_i$ to each $S_i$ is a rectangular zee-bridge, then $\mathcal C$ is flat and the angle sum at each boundary vertex is $\pi$.
 \end{lemma}
 
 \begin{proof}
  Since the restrictions to $S_i$ and $S_{i+1}$ are rectangular, the angle sums above and below their horizontal line of intersection at $a_i$ and $c_i$ are each $\pi$, so the total angle sum is $2\pi$. Similarly, since the angles above and below the horizontal line of intersection at $w_i$ and $y_i$ are each $\frac{\pi}{2}$, the total angle sum is $\pi$.
 \end{proof}
 
\subsection{$n$-Stars}
 
 While much of what follows works for more general polygons, we'll focus on a class of polygons we call ``$n$-stars''.
 
  \begin{definition}
    An \emph{$n$-star} is a planar equilateral $2n$-gon in $\R^3$ with side length 1 whose interior angles alternate between two values.
  \end{definition}
  
 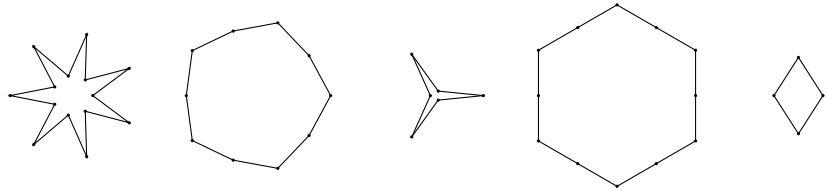
\begin{figure}
  \centering
  \begin{tikzpicture}[scale=0.6]
   \def\n{7}
   \def\p{5}
   \def\a{{sin(\p/2 r)/sin(pi/\n r)}}
   \def\b{{sin((pi-pi/\n-\p/2) r)/sin(pi/\n r)}}
   
   \foreach \i in {1,2,...,\n}
    {
      \draw ({\i*360/\n}:\b) -- ({(\i+0.5)*360/\n}:\a) -- ({(\i+1)*360/\n}:\b);
      \fill ({\i*360/\n}:\b) circle[radius=1pt];
      \fill ({(\i+0.5)*360/\n}:\a) circle[radius=1pt];
    }
    
   \begin{scope}[shift={(4,0)}]
   \def\n{5}
   \def\p{2.15}
   \def\a{{sin(\p/2 r)/sin(pi/\n r)}}
   \def\b{{sin((pi-pi/\n-\p/2) r)/sin(pi/\n r)}}
   
   \foreach \i in {1,2,...,\n}
    {
      \draw ({\i*360/\n}:\b) -- ({(\i+0.5)*360/\n}:\a) -- ({(\i+1)*360/\n}:\b);
      \fill ({\i*360/\n}:\b) circle[radius=1pt];
      \fill ({(\i+0.5)*360/\n}:\a) circle[radius=1pt];
    }
    
    \end{scope}
    
    \begin{scope}[shift={(12,0)}]
   \def\n{6}
   \def\p{pi}
   \def\a{{sin(\p/2 r)/sin(pi/\n r)}}
   \def\b{{sin((pi-pi/\n-\p/2) r)/sin(pi/\n r)}}
   
   \foreach \i in {1,2,...,\n}
    {
      \draw ({\i*360/\n}:\b) -- ({(\i+0.5)*360/\n}:\a) -- ({(\i+1)*360/\n}:\b);
      \fill ({\i*360/\n}:\b) circle[radius=1pt];
      \fill ({(\i+0.5)*360/\n}:\a) circle[radius=1pt];
    }
    
    \end{scope}
    
    \begin{scope}[shift={(16,0)}]
   \def\n{2}
   \def\p{2}
   \def\a{{sin(\p/2 r)/sin(pi/\n r)}}
   \def\b{{sin((pi-pi/\n-\p/2) r)/sin(pi/\n r)}}
   
   \foreach \i in {1,2,...,\n}
    {
      \draw ({\i*360/\n}:\b) -- ({(\i+0.5)*360/\n}:\a) -- ({(\i+1)*360/\n}:\b);
      \fill ({\i*360/\n}:\b) circle[radius=1pt];
      \fill ({(\i+0.5)*360/\n}:\a) circle[radius=1pt];
    }
    
    \end{scope}
    
    \begin{scope}[shift={(8,0)}]
   \def\n{3}
   \def\p{0.2}
   \def\a{{sin(\p/2 r)/sin(pi/\n r)}}
   \def\b{{sin((pi-pi/\n-\p/2) r)/sin(pi/\n r)}}
   
   \foreach \i in {1,2,...,\n}
    {
      \draw ({\i*360/\n}:\b) -- ({(\i+0.5)*360/\n}:\a) -- ({(\i+1)*360/\n}:\b);
      \fill ({\i*360/\n}:\b) circle[radius=1pt];
      \fill ({(\i+0.5)*360/\n}:\a) circle[radius=1pt];
    }
    
    \end{scope}
    
  \end{tikzpicture}~

  \caption{Some $n$-stars.}
  
  \end{figure}
 
 Since the sum of the interior angles of an embedded $2n$-gon must be $2\pi(n-1)$, the two interior angles of an $n$-star must be $\phi$ and $\left(\frac{2\pi(n-1)}{n}\right)-\phi$ for some $0<\phi<\frac{2\pi(n-1)}{n}$. Note that when $\phi=\pi$, the $n$-star is actually a regular $n$-gon with side length 2 with a vertex at the midpoint of each edge.
 
  
  
  Every $n$-star $\mathcal P=[P_0,\ldots,P_{2n-1}]$ has a unique \emph{center} $\tilde{\mathcal P}$ that is equidistant from all of the even numbered vertices and equidistant from all of the odd numbered vertices. Note that the angle between $\overrightarrow{\tilde{\mathcal P} P_i}$ and $\overrightarrow{\tilde{\mathcal P} P_{i+1}}$ is $\frac{\pi}{n}$. The line through $\tilde{\mathcal P}$ normal to the plane containing $\mathcal P$ is called the \emph{axis} of $\mathcal P$.
  
 \begin{definition}
  If $n$ is even and $\mathcal P=[P_0,\ldots,P_{2n-1}]$ and $\mathcal Q=[Q_0,\ldots,Q_{2n-1}]$ are $n$-stars, we say $(\mathcal P, \mathcal Q,\hat s, \hat t)$ is an \emph{aligned $n$-star tube frame} when the axes of $\mathcal P$ and $\mathcal Q$ intersect, and there is a plane containing both axes and the points $P_0$ and $Q_0$.
 \end{definition}
 
 Aligned $n$-star tube frames have some special symmetry properties that will allow us to generate flat tube connections from a small number of parameters. 
 
 \subsection{Generating Flat Tube Connections}\label{sec:generating}
  
 Let $\mathscr F$ be an aligned $n$-star tube frame, and define the bridge frames $\mathcal F_i$ as in (\ref{eq:fi}). We seek conditions under which a tube connection $\mathcal C_{\mathscr{F},\vec\alpha,\vec\gamma}$ will be flat. The basic idea is that if we know the parameters $\alpha_k,\gamma_k$ for some fixed $k$, we can apply the functions $\bar\beta,\bar\delta$ from (\ref{eq:bars}) to $\mathcal F_{k+1}$ to find values for $\alpha_{k+1},\gamma_{k+1}$ that will make the restriction $\mathcal Z_i$ rectangular if the discriminant is positive. By applying this procedure recursively and leveraging the symmetry of the tube frame, we can ensure that each of the zee-bridges $\mathcal Z_i$ is rectangular and apply Lemma \ref{lem:rectangluar-zee-bridge-to-flat-tubes}.
 
 Specifically, suppose $\alpha,\gamma>0$ are fixed and $\mathscr F$ is an aligned $n$-star tube frame. Let $\alpha_{0}=\alpha$ and $\gamma_{0}=\gamma$. Then for $i=0,\ldots,n-1$, recursively set
 \begin{align*}
  \alpha_{i+1}&=\bar\delta(\mathcal F_{i+1},\alpha_{i},\gamma_{i})\\
  \gamma_{i+1}&=\bar\beta(\mathcal F_{i+1},\alpha_{i},\gamma_{i}),
 \end{align*}
 and for $i=1,\ldots,n-1$, set
 $$\alpha_{2n-i}=\alpha_{i} \qquad \gamma_{2n-i}=\gamma_{i}.$$
 Then we say $\vec \alpha=[\alpha_0,\ldots,\alpha_{2n-1}]$ and $\vec \gamma = [\gamma_0,\ldots,\gamma_{2n-1}]$ are the parameter sets for $\mathscr F$ \emph{generated by} $\alpha$ and $\gamma$. 
 
 \begin{theorem}\label{thm:aligned-tube-joint-consistency}
  Suppose $\mathscr F$ is an aligned $n$-star tube frame and suppose $\alpha,\gamma$ generate positive parameter sets $\vec \alpha,\vec \gamma$ for $\mathscr F$. Let $\mathcal C=(T,\Phi^0)=\mathcal C_{\mathscr F,\vec\alpha,\vec\gamma}$.  If for all $i=1,\ldots,n$,
  $$\Delta_i:=\alpha_{i-1}-\gamma_{i}+|A_{i-1}C_{i-1}|+\gamma_{i-1}-\alpha_{i}>0$$
  where $A_i=\Phi^0(a_i)$ and $C_i=\Phi^0(c_i)$, then $\mathcal C$ is flat.
 \end{theorem}
 
 \begin{proof}  
  Note that for $i=1,\ldots,n$, the restriction $\mathcal Z_i$ of $\mathcal C$ to $S_{i}$ is generated by $\alpha_{i-1},\gamma_{i-1}$ for the bridge frame $\mathcal F_{i}$. Furthermore $\Delta_i$ is defined so that it is the discriminant of $\mathcal Z_{i}$, and thus if $\Delta_i>0$, then $\mathcal Z_{i}$ is rectangular. This proves that each of $\mathcal Z_{1},\ldots,\mathcal Z_{n}$ is a rectangular zee-bridge.
  
  For the remaining zee-bridges, we will use symmetry. Let $\rho$ denote the affine reflection map across the plane containing the axes of $\mathcal P$ and $\mathcal Q$ and the points $P_{0}$ and $Q_{0}$. The restriction of $\rho$ to the plane containing $\mathcal P$ is then a reflection across the line through $P_{0}$ and $\tilde{\mathcal P}$, so it is clear that for any $i$, $\rho(P_{i})=P_{2n-i}$; similarly $\rho(Q_{i})=Q_{2n-i}$.
  
  This means that the bridge frame $\tilde{\mathcal F}_{2n-i}$ is exactly the image under $\rho$ of $\mathcal F_{i+1}$. Thus for $i=0,\ldots,n-1$, the zee-bridge for $\tilde{\mathcal F}_{2n-i}$ with the same parameters as $\mathcal Z_{i+1}$ will also be rectangular.
  
  Since the restriction of $\mathcal C$ to any $S_i$ is rectangular, by Lemma \ref{lem:rectangluar-zee-bridge-to-flat-tubes}, $\mathcal C$ is flat.
 \end{proof}
 
 We now specialize further to the types of aligned $n$-star tube frames that will be used to build the main construction.
 
 Let $n\geq 4$ be even, $L>0$, $0\leq \theta\leq \pi$, and $0<\psi<\frac{2\pi(n-1)}{n}$. Let $\mathcal P=[P_0,\ldots,P_{2n-1}]$ be the $n$-star centered around $(L,0,0)$ in the plane to which $\hat s=(-1,0,0)$ is normal so that $\dvec{\tilde{\mathcal P}P_{n/2}}$ has direction $\hat k=(0,0,1)$, $\angle P_0=\pi$ (so that $\mathcal P$ is an $n$-gon with side length 2), and $(\mathcal P,\hat s)$ is right handed. Let $\mathcal Q=[Q_0,\ldots,Q_{2n-1}]$ be an $n$-star centered around $(L\cos\theta,L\sin\theta,0)$ in the plane to which $\hat t=(\cos\theta,\sin\theta,0)$ is normal so that $\dvec{\tilde{\mathcal Q}Q_{n/2}}$ has direction $\hat k$, $\angle Q_0=\psi$, and $(\mathcal Q,\hat t)$ is left-handed. In this case we write $\mathscr J_{n,L,\theta,\psi}$ for the tube frame $(\mathcal P,\mathcal Q,\hat s,\hat t)$, which is an aligned $n$-star frame since the axes of $\mathcal P$ and $\mathcal Q$ intersect at the origin and lie in the $xy$-plane along with $P_{0}$ and $Q_{0}$.
  
 Under the same restrictions for $n,L,\psi$, let $\mathcal P$ be as above, and now let $\mathcal Q=[Q_0,\ldots,Q_{2n-1}]$ be the $n$-star centered around $(-L,0,0)$ in the plane to which $\hat s$ is normal so that $\dvec{\tilde QQ_{n/2}}$ points in direction $\hat k$, $\angle Q_0=\psi$, and $\mathcal Q$ is right-handed. In this case, we write $\mathscr S_{n,L,\psi}$ for the tube frame $(\mathcal P,\mathcal Q,\hat s,\hat s)$, which is also an aligned $n$-star tube frame for the same reason.
 
 Examples of tube frames of the forms $\mathscr J_{n,L,\theta,\psi}$ and $\mathscr S_{n,L,\psi}$ are shown in Figure \ref{fig:j-frame}.
 
    \begin{figure}
   \centering
   \begin{subfigure}[b]{0.45\textwidth}
    \centering
    \begin{overpic}[width=\textwidth,tics=10]{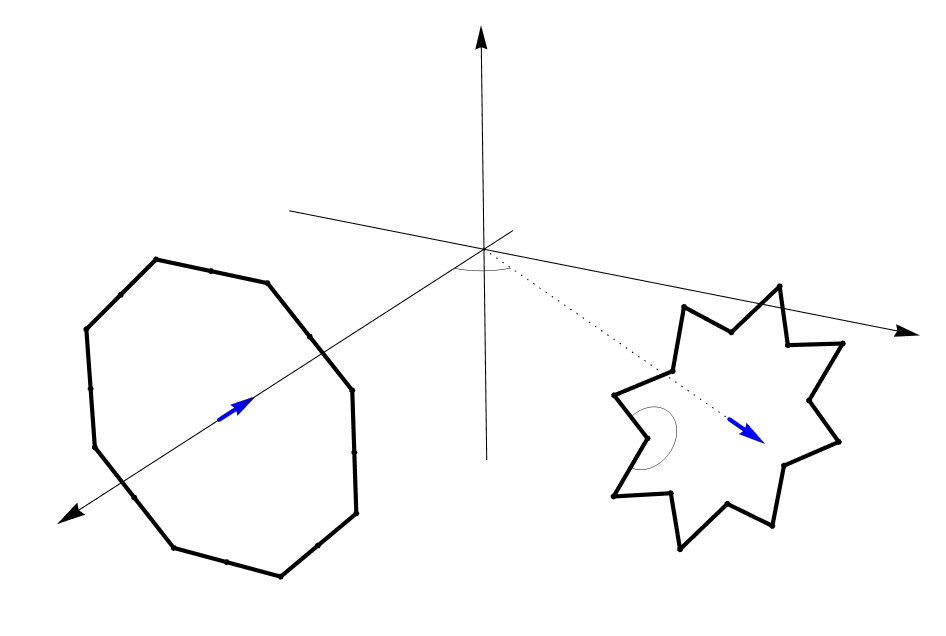}
 \put (3,24) {\small$P_0$}
 \put (21,40) {\small$P_{n/2}$}
 \put (39,18) {\small$P_{n}$}
 \put (18,4) {\small$P_{3n/2}$}
 \put (69,37) {\small$Q_{n/2}$}
 \put (60,20) {\small$Q_{0}$}
 \put (75,8) {\small$Q_{3n/2}$}
 \put (87,24) {\small$Q_n$}

 \put (23,25) {\small\color{blue}$\hat s$}
 \put (79,24) {\small\color{blue}$\hat t$}
 
 \put (52,35) {\small$\theta$}
 \put (72,20) {\small$\psi$}

\end{overpic}
    \caption{}
    \label{fig:joint}
  \end{subfigure}
  \hfill 
  \begin{subfigure}[b]{0.45\textwidth}
    \centering
    \begin{overpic}[width=0.8\textwidth,tics=10]{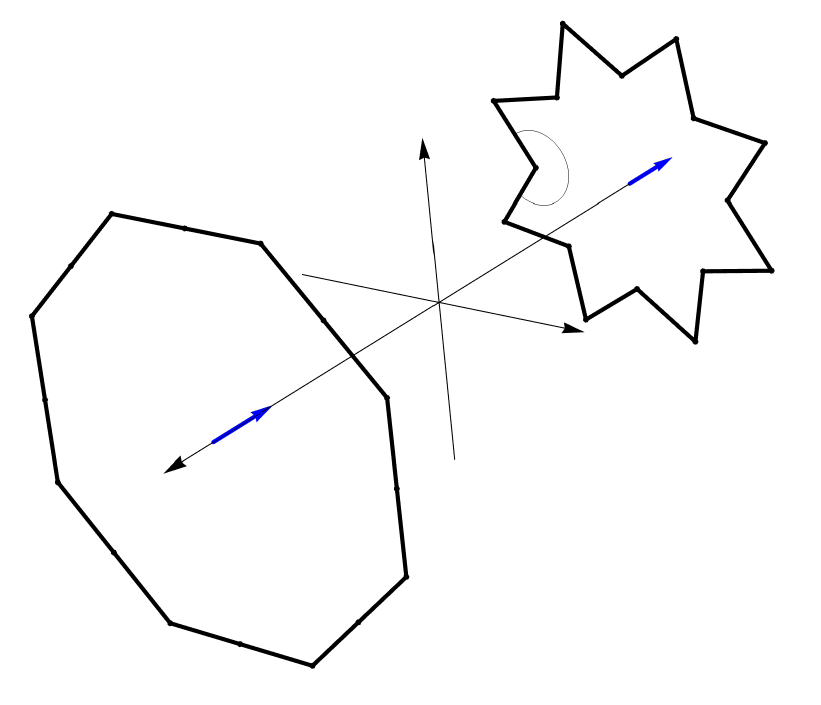}
 \put (-3,33) {\small$P_0$}
 \put (20,60) {\small$P_{n/2}$}
 \put (50,24) {\small$P_{n}$}
 \put (18,3) {\small$P_{3n/2}$}
 
 \put (71,83) {\small$Q_{n/2}$}
 \put (54,63) {\small$Q_{0}$}
 \put (72,42) {\small$Q_{3n/2}$}
 \put (92,58) {\small$Q_n$}

 \put (27,37) {\small\color{blue}$\hat s$}
 \put (78,68) {\small\color{blue}$\hat t$}
 
 \put (69,67) {\small$\psi$}

\end{overpic}
    \caption{}
    \label{fig:straight}
  \end{subfigure}
      
   \caption{(a) A tube frame of the form $\mathscr J_{n,L,\theta,\psi}$. (b) A tube frame of the form $\mathscr S_{n,L,\psi}$.}\label{fig:j-frame}
  \end{figure}
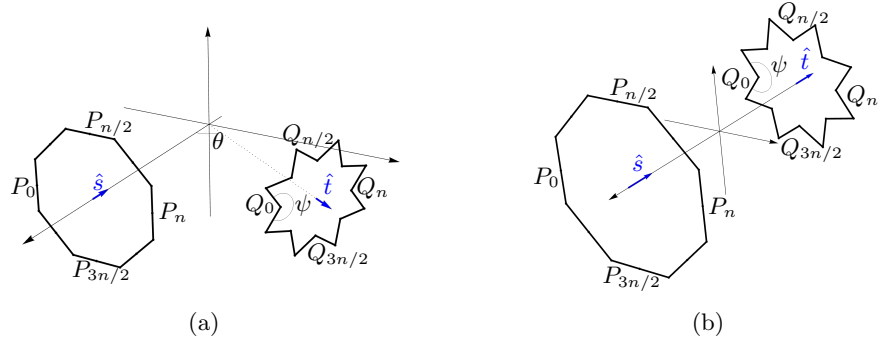
  
\subsection{Examples of Flat Tube Connections}
  
 Our plan is to glue together immersed flat tube connections over tube frames of the form $\mathscr J_{n,L,\theta,\psi}$ and $\mathscr S_{n,L,\psi}$. The following two theorems prove flatness for the specific cases we'll need.

 
 \begin{theorem}\label{thm:straight-flat}
  Suppose $n\geq4$ is even, $L>0$, and $\pi<\psi<\frac{2\pi(n-1)}{n}$, and let $\mathscr S=\mathscr S_{n,L,\psi}$.   
   Then any $\alpha,\gamma>0$ generate positive parameter sets $\vec \alpha,\vec \gamma$ for $\mathscr S$ and the tube connection $\mathcal S=\mathcal C_{\mathscr S,\vec\alpha,\vec\gamma}$ is flat with
  $$\alpha_1=\alpha_3=\cdots=\alpha_{2n-1}, \quad \gamma_0=\gamma_2=\cdots=\gamma_{2n-2}, \text{ and}$$
  $$\alpha_0=\gamma_1=\alpha_2=\gamma_3=\cdots=\gamma_{2n-1}.$$
 \end{theorem}

%
%
%
 
 \begin{proof}
  We will start by describing the restriction $\mathcal Z_1$ of $\mathcal S$ to $S_1$, and then use symmetry to get the desired result.
  
  If $\mathscr S=(\mathcal P,\mathcal Q,\hat s,\hat s)$, using the construction from Section \ref{sec:generating}, let $\alpha_0=\alpha$, $\gamma_0=\gamma$, and
  $$\gamma_1=\bar\beta(\mathcal F_1,\alpha_0,\gamma_0)=\alpha_0+\frac{\dvec{AC}\cdot \hat u}{|AC|-\dvec{AC}\cdot(-\hat s)}$$
  where $A=P_0+\alpha_0\hat s$, $C=Q_0-\gamma_0\hat s$, and $\hat u=\dvec{P_0P_1}=(0,0,1)$. Since $\psi>\pi$, the right cylinder over $\mathcal Q$ lies entirely inside of the right cylinder over $\mathcal P$, so $\dvec{AC}$ is not parallel to $\hat s$. Thus the denominator is nonzero, and so $\gamma_1$ above is well-defined. Furthermore, $P_0$ and $Q_0$ lie in the $xy$-plane, so $\dvec{AC}\cdot\hat u=0$, and thus $\gamma_1=\alpha_0$.
  
  Also the construction has us set
  $$\alpha_1=\bar\delta(\mathcal F_1,\alpha_0,\gamma_0)=\gamma_0+\frac{\dvec{BB'}\cdot\hat v}{\dvec{BB'}\cdot(-\hat s)}$$
  where $B=P_1+\gamma_1\hat s$, $B'=Q_1-(\alpha_0-\gamma_1+|AC|+\gamma_0)\hat s$, and $\vec v=\dvec{Q_0Q_1}$.
%
  As in the proof of Lemma \ref{lem:straight-implies-rect}, the assumption that $\hat s=\hat t$ implies that 
  $$\dvec{B'B}\cdot\hat s=\dvec{B'D}\cdot\hat s - \dvec{BD}\cdot\hat s>0.$$
  Letting $\Pi$ be projection onto the $yz$-plane, to which $\hat s$ is normal and $\hat v$ is parallel, we have $\dvec{B'B}\cdot \hat v=\Pi\left(\dvec{B'B}\right)\cdot \hat v$. Note that $\Pi\left(\dvec{B'B}\right)=\Pi\left(\dvec{Q_1P_1}\right)$ since $B'$ and $B$ project to the same points as $Q_1$ and $P_1$ respectively. We can then find that the angle between $\hat v=\dvec{Q_0Q_1}$ and $\Pi\left(\dvec{Q_1P_1}\right)$ is $\pi-\frac{\psi}{2}-\frac{\pi}{n}<\frac{\pi}{2}$ as in Figure \ref{fig:coaxial}, so this dot product is negative, and thus $\alpha_1>\gamma_0>0$. This shows that $\alpha,\gamma$ generate positive parameters for $\mathcal F_1$, so by Lemma \ref{lem:straight-implies-rect}, $\mathcal Z_1$ is rectangular.
  
  \begin{figure}
  \centering
  
  \begin{subfigure}[b]{0.45\textwidth}
    \centering
    \begin{tikzpicture}[scale=1]

   \def\n{6}
   \def\p{pi}
   \def\a{{sin(\p/2 r)/sin(pi/\n r)}}
   \def\b{{sin((pi-pi/\n-\p/2) r)/sin(pi/\n r)}}
   
   \foreach \i in {0,...,5}
    {
      \draw ({\i*360/\n}:\b) -- ({(\i+0.5)*360/\n}:\a) -- ({(\i+1)*360/\n}:\b);
      \fill ({\i*360/\n}:\b) circle[radius=1pt];
      \fill ({(\i+0.5)*360/\n}:\a) circle[radius=1pt];
    }
   
   \foreach \j in {0,2,...,10}
    {
     \draw ({\j*180/\n}:2.1) node {$ P_{\j}$};
     \draw ({\j*180/\n}:0.85) node {$ Q_{\j}$};
    }
    
   \foreach \j in {1,3,...,11}
    {
     \draw ({\j*180/\n}:2.3) node {$ P_{\j}$};
     \draw ({\j*180/\n}:1.64) node {$ Q_{\j}$};
    }
    
   \def\n{6}
   \def\p{4.8}
   \def\a{{sin(\p/2 r)/sin(pi/\n r)}}
   \def\b{{sin((pi-pi/\n-\p/2) r)/sin(pi/\n r)}}
   
   \foreach \i in {1,2,...,\n}
    {
      \draw ({\i*360/\n}:\b) -- ({(\i+0.5)*360/\n}:\a) -- ({(\i+1)*360/\n}:\b);
      \fill ({\i*360/\n}:\b) circle[radius=1pt];
      \fill ({(\i+0.5)*360/\n}:\a) circle[radius=1pt];
    }
    \end{tikzpicture}
    \caption{}
    \label{fig:coaxial-frame}
  \end{subfigure}
  \hfill 
  \begin{subfigure}[b]{0.45\textwidth}
    \centering
    \begin{tikzpicture}[scale=1]

     \def\n{6}
   \def\q{pi}
   \def\a{{sin(\q/2 r)/sin(pi/\n r)}}
   \def\b{{sin((pi-pi/\n-\q/2) r)/sin(pi/\n r)}}
   
   \def\p{4.8}
   \def\c{{sin(\p/2 r)/sin(pi/\n r)}}
   \def\d{{sin((pi-pi/\n-\p/2) r)/sin(pi/\n r)}}
   
   \foreach \i in {0,...,5}
    \fill[gray!20] ({(\i-0.5)*360/\n}:\a) -- ({(\i+0.5)*360/\n}:\a) -- ({(\i+0.5)*360/\n}:\c) -- ({(\i)*360/\n}:\d) -- ({(\i-0.5)*360/\n}:\c);
   
     \foreach \i in {0,...,5}
    {
      \draw ({\i*360/\n}:\b) -- ({(\i+0.5)*360/\n}:\a) -- ({(\i+1)*360/\n}:\b);
    }

%

   \foreach \i in {1,2,...,\n}
    {
      \draw ({\i*360/\n}:\d) -- ({(\i+0.5)*360/\n}:\c) -- ({(\i+1)*360/\n}:\d);
    }
    
    \foreach \i in {0,...,5}
     {
      \draw ({\i*360/\n}:\b) -- ({(\i)*360/\n}:\d) -- ({(\i+0.5)*360/\n}:\a) -- ({(\i+0.5)*360/\n}:\c);
      \draw ({(\i+0.5)*360/\n}:\a) -- ({(\i+1)*360/\n}:\d);
     }
     
     \foreach \j in {1,3,...,11}
    {
     \draw ({\j*180/\n}:2.3) node {$\phantom{Q_{\j}}$};
    }

  \end{tikzpicture}
    \caption{}
    \label{fig:str-jt-top-view}
  \end{subfigure}

  \caption{(a) A tube frame of the form $\mathscr S=\mathscr S_{n,L,\psi}$ projected onto the $yz$-plane. (b) A tube connection for $\mathscr S$ projected onto the $yz$-plane.}\label{fig:coaxial}
 \end{figure}
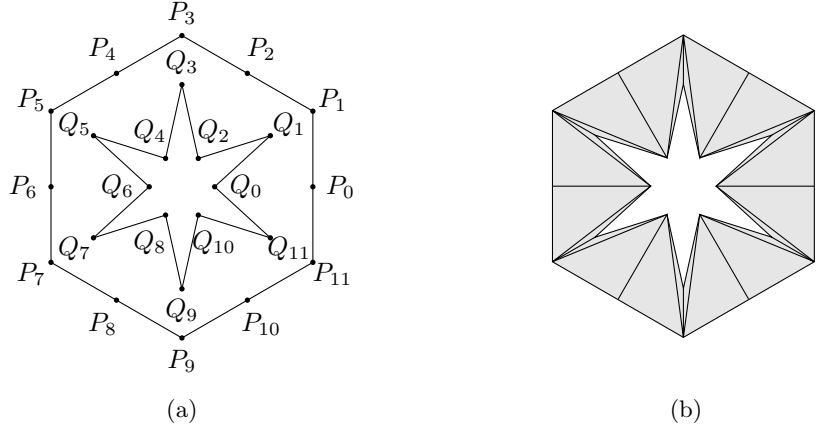
  
  Because of the dihedral symmetry of $n$-stars, if we set the rest of the $\alpha_i$ and $\gamma_i$ to be as in the theorem, all of the $\mathcal Z_i$ will be rectangular, and by the uniqueness parts of Lemmas \ref{lem:b} and \ref{lem:d}, $\vec\alpha,\vec\gamma$ will be the parameter sets generated by $\alpha,\gamma$. Thus by Lemma \ref{lem:rectangluar-zee-bridge-to-flat-tubes}, $\mathcal S$ is flat.
 \end{proof}
 
 An example of a tube connection over $\mathscr S_{n,L,\psi}$ can be found in Figure \ref{fig:str-jt-example}. The equality of parameters in a coaxial tube joint allows us to identify adjacent coplanar faces and collinear edges that can be merged for simplicity.

 Now we prove our first theorem regarding the specific parameters used in Theorem A.
 
 \begin{theorem}\label{thm:example-flat}
  Let $\mathscr J=\mathscr J_{6,4,\frac{\pi}{3},\frac{3\pi}{2}}$, and let $\vec\alpha,\vec\gamma$ be the parameter sets generated by $\alpha=3.1,\gamma=2.5$. Then $\mathcal C_{\mathscr J,\vec\alpha,\vec\gamma}$ is flat.
 \end{theorem}
 
 \begin{proof}
  Here we apply Theorem \ref{thm:aligned-tube-joint-consistency}, by checking the positivity of $\alpha_i$, $\gamma_i$, and $\Delta_i$. All calculations were done in Mathematica using interval arithmetic to certify the positivity of these values. Specifically, for each value, we created an interval of radius $10^{-20}$ containing each input parameter, and calculated an interval containing the desired value using outward-rounding interval arithmetic; the positivity of a value is certified when this interval contains only positive numbers. The radii of the intervals around the values were all less than $10^{-7}$. In Table \ref{tab:flatness-data}, we round the center of each interval to 5 decimal places so that that error in these approximations is less than $10^{-5}$.
  \begin{table}
   \centering

\scriptsize
\begin{subtable}[b]{0.45\textwidth}
    \centering
\begin{tabular}{| c | T | T |}
 \hline
 $i$ & \multicolumn{1}{c|}{$\alpha_i$} & \multicolumn{1}{c|}{$\gamma_i$}\\
 \hline
 0 & 3.1 & 2.5 \\
 1 & 4.61771 & 3.1 \\
 2 & 3.52866 & 2.59407 \\
 3 & 3.82395 & 3.95732 \\
 4 & 4.37423 & 2.55177 \\
 5 & 3.10927 & 4.79114 \\
 6 & 4.79114 & 2.48504 \\
 \hline
\end{tabular}
    \caption{}
    \label{fig:parameters-positive}
  \end{subtable}
  \hfill 
  \begin{subtable}[b]{0.45\textwidth}
    \centering
\begin{tabular}{|c|T|}
\hline
 $j$ & \multicolumn{1}{c|}{$\Delta_{j}$}\\
\hline
 1 & 0.67056 \\
 2 & 2.26554 \\
 3 & 0.607 \\
 4 & 1.46227 \\
 5 & 0.48786 \\
 6 & 1.11209 \\
\hline
\end{tabular}
    \caption{}
    \label{fig:delta-i-positive}
  \end{subtable}
   
   \caption{(a) The parameters generated by $\alpha=3.1$, $\gamma=2.5$ for $\mathscr J_{6,4,\frac{\pi}{3},\frac{3\pi}{2}}$. (b) The values $\Delta_j$ (as defined in Theorem \ref{thm:aligned-tube-joint-consistency}).}\label{tab:flatness-data}
  \end{table}%
 \end{proof}
 
 The image of $\mathcal C_{\mathscr J,\vec\alpha,\vec\gamma}$ from Theorem \ref{thm:example-flat} is shown in Figure \ref{fig:example-joint} along with its combinatorial structure.
 
 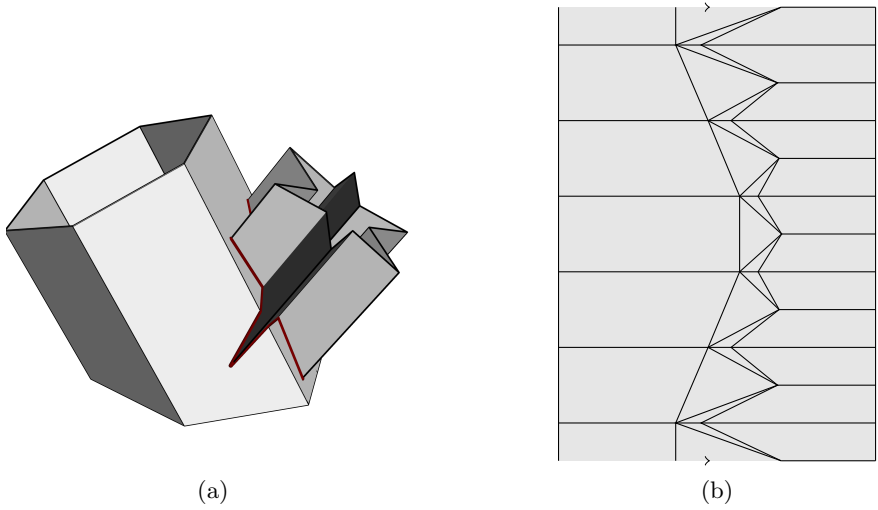
\begin{figure}
   \centering
   \begin{subfigure}[b]{0.45\textwidth}
    \centering
    \includegraphics[width=\textwidth]{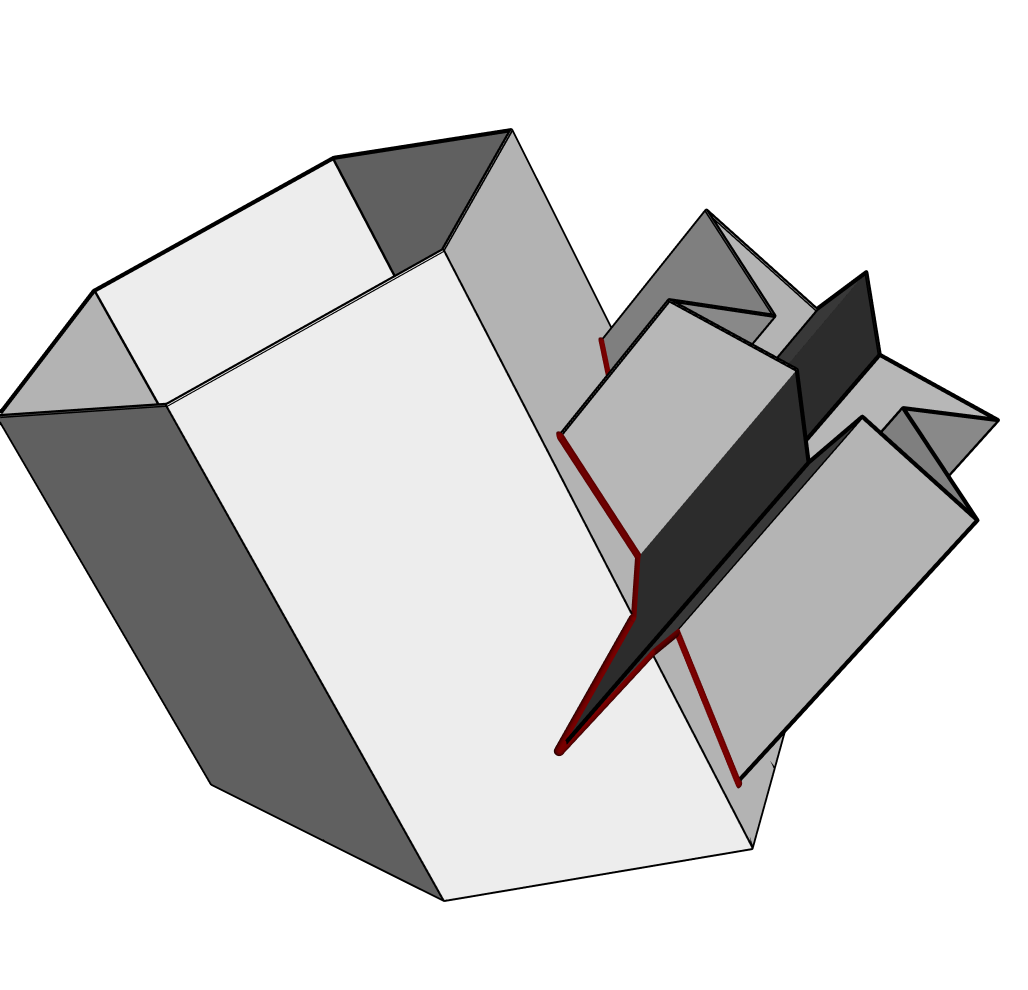}
    \caption{}
    \label{fig:example-joint-image}
  \end{subfigure}
  \hfill 
  \begin{subfigure}[b]{0.45\textwidth}
    \centering

    \begin{tikzpicture}[scale=0.5]
     \def\w{8.38827}
     
     \def\aa{3.52866}
     \def\ab{4.61771}
     \def\ac{3.1}
     \def\ad{4.61771}
     \def\ae{3.52866}
     \def\af{3.82395}
     \def\ag{4.37423}
     \def\ah{3.10927}
     \def\ai{4.79114}
     \def\aj{3.10927}
     \def\ak{4.37423}
     \def\al{3.82395}
     
     \def\ca{2.59407}
     \def\cb{3.1}
     \def\cc{2.5}
     \def\cd{3.1}
     \def\ce{2.59407}
     \def\cf{3.95732}
     \def\cg{2.55177}
     \def\ch{4.79114}
     \def\ci{2.48504}
     \def\cj{4.79114}
     \def\ck{2.55177}
     \def\cl{3.95732}
     
     \fill[gray!20] (0,3) rectangle (\w,15);
     \draw (0,3) -- (0,15);
     \draw (\w,3) -- (\w,15);
     \draw ({\w-\cc},15) -- (\w,15);
     
     \draw[->] (3.99,3) -- (4,3);
     \draw[->] (3.99,15) -- (4,15);
     
     \foreach \i in {4,6,...,14} {\draw (0,\i) -- (\w,\i);}
     \draw ({\w-\ca},13) -- (\w,13);
     \draw ({\w-\cc},3) -- (\w,3);
     \draw ({\w-\ce},5) -- (\w,5);
     \draw ({\w-\cg},7) -- (\w,7);
     \draw ({\w-\ci},9) -- (\w,9);
     \draw ({\w-\ck},11) -- (\w,11);
     
      
     \draw (\ac,3) -- (\cd,4) -- (\ae,5) -- (\cf,6) -- (\ag,7) -- (\ch,8) -- (\ai,9) -- (\cj,10) -- (\ak,11) -- (\cl,12) -- (\aa,13) -- (\cb,14) -- (\ac,15);
     \draw ({\w-\cc},3) -- (\cd,4) -- ({\w-\ce},5) -- (\cf,6) -- ({\w-\cg},7) -- (\ch,8) -- ({\w-\ci},9) -- (\cj,10) -- ({\w-\ck},11) -- (\cl,12) -- ({\w-\ca},13) -- (\cb,14) -- ({\w-\cc},15);
     \draw ({\w-\cc},3) -- ({\w-\ad},4) -- ({\w-\ce},5) -- ({\w-\af},6) -- ({\w-\cg},7) -- ({\w-\ah},8) -- ({\w-\ci},9) -- ({\w-\aj},10) -- ({\w-\ck},11) -- ({\w-\al},12) -- ({\w-\ca},13) -- ({\w-\ab},14) -- ({\w-\cc},15);
     
    \end{tikzpicture}
    
    \caption{}
    \label{fig:example-joint-domain}
  \end{subfigure}
   
   \caption{(a) The image of the polyhedral surface $\mathcal C_{\mathscr J,\vec\alpha,\vec\gamma}$ from Theorem \ref{thm:example-flat}. The self-intersection is highlighted in maroon; the triangular faces are obscured. (b) The combinatorial structure of $\mathcal C_{\mathscr J,\vec\alpha,\vec\gamma}$ with the induced length structure. Adjacent coplanar faces and collinear edges have been merged.}
   \label{fig:example-joint}
  \end{figure}%
 
\section{Immersed Vertices} \label{sec:local-injectivity}
 
 Up until now, we have made no assumptions about how faces of zee-bridges and tube connections intersect. For example, cross-cap vertices and dihedral angles of zero have thus far been allowed. While no embedded polyhedral Klein bottle in $\R^3$ exists, we are aiming for an immersed one---one where the star of each vertex is embedded.

 Let $u$ be a vertex of a polyhedral surface. If $f_1,f_2$ are faces whose intersection is exactly $\{u\}$, we say $f_1,f_2$ \emph{meet kitty-corner} at $u$.
 
 \begin{lemma}\label{lem:kitty}
  Let $\mathcal S$ be a polyhedral surface, and let $u$ be an interior vertex at which 4 or more edges meet. Then $\mathcal S$ is immersed at $u$ if and only if the images of every pair of faces meeting kitty-corner at $u$ intersect only at the image of $u$.
 \end{lemma}
 
 \begin{proof}
  Suppose $\mathcal S$ is not immersed at $u$, meaning that either the images of two faces meeting kitty-corner at $u$ intersect away from the image of $u$ (in which case we are done), or the images of two faces  that contain both $u$ and some other vertex $v$ intersect at some point $x$ not on the edge connecting the images of $u$ and $v$. In the latter case, the images of the faces $f$ and $g$ are convex planar polygons containing the (non-collinear) images of $u$, $v$, and $x$, so the intersection of the images of $f$ and $g$ is some convex planar polygon $\mathcal P$, one of whose boundary edges is the image of the edge connecting $u$ and $v$. The other edge of $\mathcal P$ from the image of $u$ must be contained in the image of an edge of a third face $h$. Thus the images of $f,g,h$ intersect pairwise, but since $u$ had valence at least 4, $h$ must meet either $f$ or $g$ kitty-corner. Thus either $f$ and $h$ or $g$ and $h$ are a pair of kitty-corner faces meeting at $u$ whose images intersect away from the image of $u$.
 \end{proof}
 
 Note that all interior vertices of the combinatorial structure $T$ have valence at least 4. Also the boundary vertices $w_i,y_i$ of any tube connection are immersed since the images of the two trapezoidal faces meeting at a boundary vertex intersect only along the image of the interior edge leaving the boundary vertex.
 
 \begin{proposition}\label{prop:str-jt-embedded}
  Every tube connection over $\mathscr S=\mathscr S_{n,L,\psi}$ is embedded.
 \end{proposition}
  
 \begin{proof}
  Consider the projection of the image of the polyhedral surface onto the $yz$-plane. For any tube connection over $\mathscr S$, we start by noting that the projections of the images of every pair of faces meeting kitty-corner at an interior vertex intersect only at the projection of the image of the point at which they meet, as we see in Figure \ref{fig:coaxial}, so these vertices are all immersed by Lemma \ref{lem:kitty}. Furthermore the projections of images of disjoint faces are disjoint, so the tube connection is embedded.
 \end{proof}
 
 
 
 
 Let $\mathcal S$ be a polyhedral surface, and suppose $u$ is a vertex of $\mathcal S$ and $f$ is a face containing $u$. Then exactly two other vertices $v,w$ lying in $f$ are connected by edges to $u$. If $\phi$ is the realization map, we call the vectors $\vec v=\dvec{\phi(u)\phi(v)},\vec w=\dvec{\phi(u)\phi(w)}$ the \emph{edge vectors} of $f$ from $u$. Note that by strict convexity, these two edge vectors are linearly independent.
 

 \begin{lemma}\label{lem:m}
  Suppose $\mathcal S$ is a polyhedral surface with extended realization map $\phi$. Let $u$ be a vertex of $\mathcal S$ lying in faces $f_1,f_2$, and let $\vec v_i,\vec w_i$ be the edge vectors of $f_i$ from $u$ for $i=1,2$. Set 
  $$M=\begin{bmatrix} \vec v_1 & \vec w_1 & -\vec v_2 & -\vec w_2 \end{bmatrix}.$$
  Then the following are equivalent.
  \begin{enumerate}
   \item The images of ${f_1}$ and ${f_2}$ intersect only at the image of $u$.
   \item There is no nonzero vector in $\ker M$ with all non-negative entries.
   \item Every nonzero vector in $\ker M$ has a positive and a negative entry.
   \item Either
    \begin{itemize}
     \item $M$ has rank 3 and there exists a vector in $\ker M$ with both a positive and a negative entry, or
     \item $M$ has rank 2 and in the kernel of every $3\times3$ submatrix of $M$ there exists a vector with both a positive and a negative entry.
    \end{itemize}
  \end{enumerate}
 \end{lemma}

 \begin{proof}
  Since the image of $ f_i$ is a strictly convex polygon, every point in $\phi(f_i)$ has the form $\phi(u)+A_i\vec v_i+B_i\vec w_i$ for some $A_i,B_i\geq0$ for $i=1,2$. Furthermore, if $A_i,B_i\geq0$, then there exists $\epsilon>0$ small enough so that $\phi(u)+\epsilon(A_i\vec v_i+B_i\vec w_i)$ lies in $\phi(f_i)$. Thus there exists $y\in\phi( f_1)\cap\phi( f_2)\setminus\{\phi(u)\}$ if and only if there exist $A_1,B_1,A_2,B_2\geq0$, not all zero, such that $A_1\vec v_1+B_1\vec w_1=A_2\vec v_2+B_2\vec w_2$, or equivalently $A_1\vec v_1+B_1\vec w_1-A_2\vec v_2-B_2\vec w_2=0$. This proves $1\Leftrightarrow2$ by contrapositive.
  
  The kernel of $M$ having no nonzero vectors with all non-negative entries is equivalent to every nonzero vector $\vec x\in\ker M$ having a negative entry, and since $-\vec x\in\ker M$ as well, $\vec x$ also has a positive entry. This proves $2\Leftrightarrow3$.
  
  To prove $3\Leftrightarrow4$, we first note that since each pair $\{\vec v_i,\vec w_i\}$ is linearly independent, $M$ has rank 2 or 3. If $M$ has rank 3, then $\ker M$ has dimension 1. Thus any nonzero vector in $\ker M$ spans $\ker M$, so if one nonzero vector has a positive and negative entry, they all do.
  
  If $M$ has rank 2, then each pair $\{\vec v_i,\vec w_i\}$ is a basis for the column space of $M$.
   The strictly convex cones generated by $\{\vec v_1,\vec w_1\}$ and $\{\vec v_2,\vec w_2\}$ then intersect if and only if a generator of one cone is contained in the other cone---that is, one of $\vec v_1,\vec w_1$ is a non-negative linear combination of $\vec v_2,\vec w_2$ or vice versa. And such a linear combination corresponds to the existence of a nonzero vector in the kernel of a $3\times3$ submatrix of $M$ with all non-negative entries, which necessarily spans this kernel, disallowing the existence of a vector with both a positive and a negative entry.
  \end{proof}

The final condition of Lemma \ref{lem:m} gives us a convenient algorithmic way of checking local injectivity through verifying inequalities.

\begin{theorem}\label{thm:example-locally-injective}
 Let $\mathscr J=\mathscr J_{6,4,\frac{\pi}{3},\frac{3\pi}{2}}$, and let $\vec\alpha,\vec\gamma$ be the parameter sets generated by $\alpha=3.1,\gamma=2.5$. Then $\mathcal C_{\mathscr J,\vec\alpha,\vec\gamma}$ is immersed.
\end{theorem}

The proof requires going through each interior vertex and applying Condition 4 from Lemma \ref{lem:m} to each pair of kitty-corner faces meeting at that vertex. We defer the full proof to Appendix \ref{sec:appendix}.
 
\section{Global Construction} \label{sec:constructions}

 We now piece together several flat immersed tube connections to form the polyhedral surface $\nabla$ referenced in Theorem A.
 
 For the remainder of this section we let $\mathscr J=\mathscr J_{6,4,\frac{\pi}{3},\frac{3\pi}{2}}$ and $\mathscr S=\mathscr S_{6,2,\frac{3\pi}{2}}$. Let $\vec\alpha,\vec\gamma$ be the parameter sets generated by $\alpha=3.1,\gamma=2.5$ for $\mathscr J$, and let $\mathcal J=\mathcal C_{\mathscr J,\vec\alpha,\vec\gamma}$. Also let $\vec \alpha',\vec\gamma'$ be the parameter sets generated by $\alpha'=1,\gamma'=1$ for $\mathscr S$, and let $\mathcal S=\mathcal C_{\mathscr S,\vec\alpha',\vec\gamma'}$.\footnote{Because of Theorem \ref{thm:straight-flat} and Proposition \ref{prop:str-jt-embedded}, the choice of $\alpha',\gamma'$ here is arbitrary; any positive values would work. We use $\alpha'=\gamma'=1$ here for simplicity and for visual clarity in Figure \ref{fig:str-jt-example}.} The tube connection $\mathcal J$ was shown in Figure \ref{fig:example-joint}, and the tube connection $\mathcal S$ is shown in Figure \ref{fig:str-jt-example}. Note that each of $\mathcal J$ and $\mathcal S$ has a boundary curve that is a regular hexagon and a boundary curve that is a $6$-star with an interior angle of $\frac{3\pi}{2}$.
 
 \begin{figure}
  \centering
  \begin{subfigure}[b]{0.45\textwidth}
    \centering
    \includegraphics[width=\textwidth]{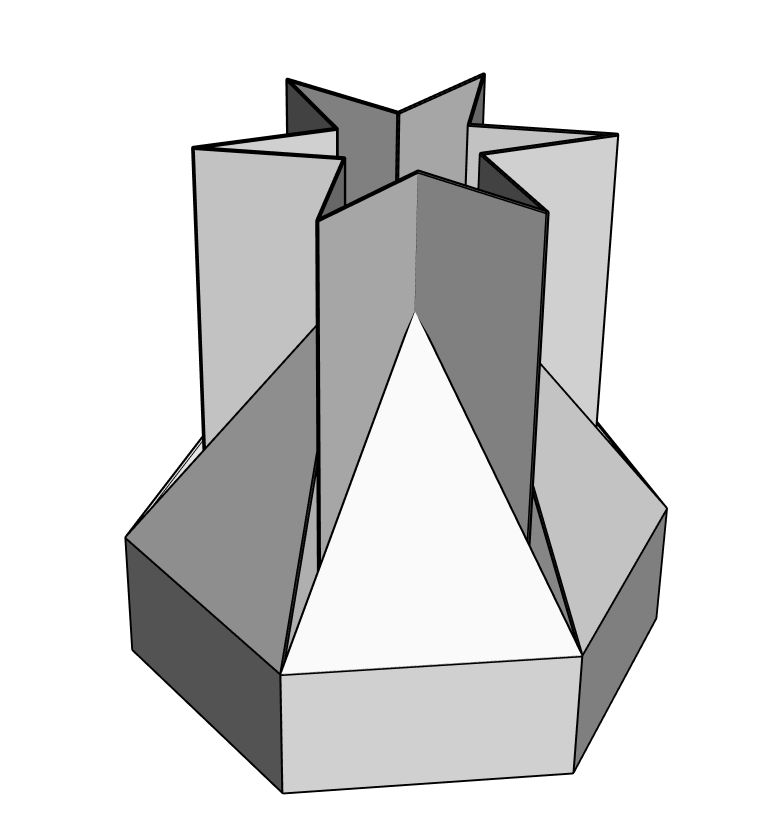}
    \caption{}
    \label{fig:str-jt-image}
  \end{subfigure}
  \hfill 
  \begin{subfigure}[b]{0.45\textwidth}
    \centering
    \begin{tikzpicture}[scale=0.5]
    \def\w{4.33983} 
    \def\h{2.33983} 
    \def\l{0.674795} 
    
    \fill[gray!20] (0,1) rectangle (\w,13);
    
    \draw (\w,13) -- (\w,1) -- ({1+\h},1);
    \draw (0,1) -- (0,13);
    
    \draw[->] ({\w/2-0.01},1) -- ({\w/2},1);
    \draw[->] ({\w/2-0.01},13) -- ({\w/2},13);

    \draw (1,1) -- (1,13);
    
    \foreach \i in {1,...,5,6}
    {
     \draw (0,{2*\i}) -- (\w,{2*\i});
     \draw ({1+\h},{2*\i-1}) -- (1,{2*\i}) -- ({1+\h},{2*\i+1});
     \draw ({1+\h},{2*\i+1}) -- ({\w},{2*\i+1});
     \draw ({1+\h},{2*\i-1}) -- ({1+\l},{2*\i}) -- ({1+\h},{2*\i+1});
    }

   \end{tikzpicture}    \caption{}
    \label{fig:str-jt-domain}
  \end{subfigure}
  
  \caption{(a) The image of a tube connection $\mathcal S$ for even $j$ from the proof of Theorem A. (b) The combinatorial structure of $\mathcal S$ with the induced length structure; adjacent coplanar faces have been merged.}\label{fig:str-jt-example}
 \end{figure}
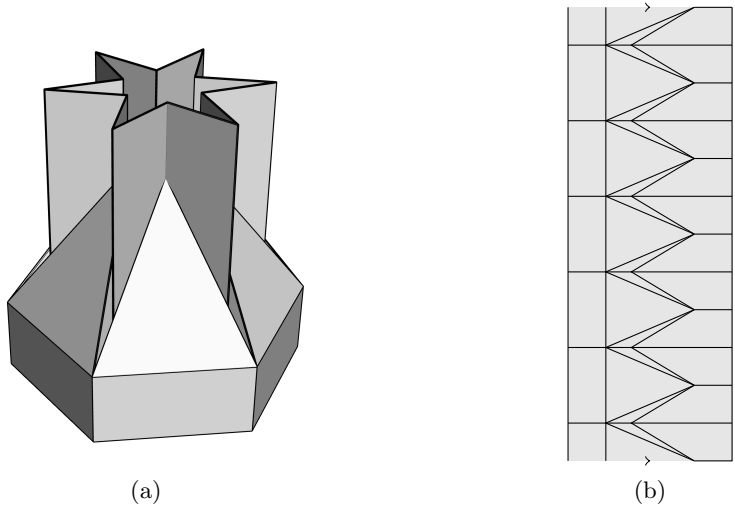

 For $j=0,1,2$ with indices taken modulo 3, let $\mathcal J^j$ be the result of starting with $\mathcal J$, first translating  by the vector $\left(-6,-2\sqrt3,0\right)$ and then rotating by an angle of $\frac{2j\pi}{3}$ around the $z$-axis counterclockwise as viewed from above. Also let $\mathcal S^j$ be the result of starting with $\mathcal S$, translating by the vector $(0,2\sqrt3,0)$, and then rotating by an angle of $\frac{(2j-1)\pi}{3}$ around the same axis. The images of the resulting six tube connections are shown in Figure \ref{fig:six-frames}.
 
 \begin{figure}
  \centering
\begin{overpic}[width=0.75\textwidth,tics=10]{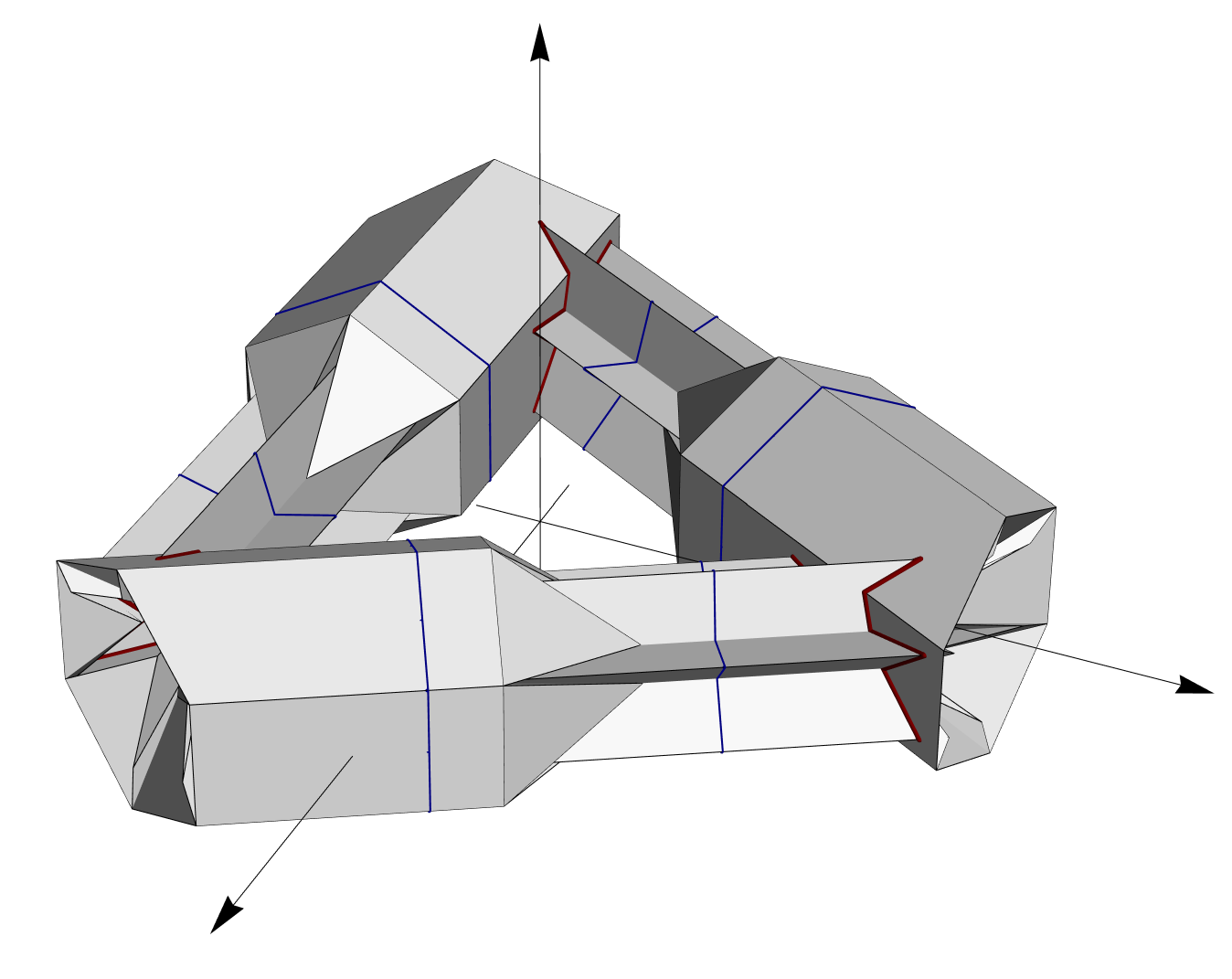}
 \put (28,65) {\small$\mathcal J^0$}
 \put (45,12) {\small$\mathcal S^0$}
 \put (14,46) {\small$\mathcal S^2$}
 \put (2,36) {\small$\mathcal J^1$}
 \put (65,51) {\small$\mathcal S^1$}
 \put (87,38) {\small$\mathcal J^2$}

\end{overpic}

  \caption{The images of the polyhedral surfaces $\mathcal J^j$ and $\mathcal S^j$ that are combined to form $\nabla$. Self-intersections are shown in maroon, and the (shared) boundary curves of the $\mathcal J^j$ and $\mathcal S^j$ are shown in blue.}
  \label{fig:six-frames}
 \end{figure}
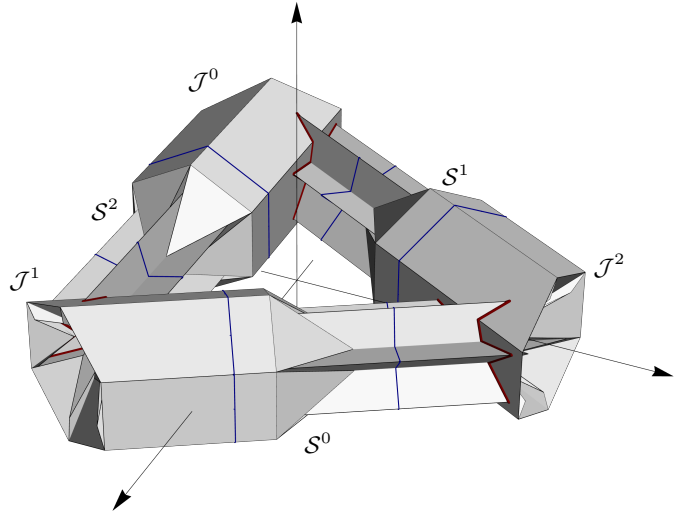
 
 Note that the hexagonal boundary curve of $\mathcal J^j$ aligns with that of $\mathcal S^{j-1}$, and the star-shaped boundary curve of $\mathcal J^j$ aligns with that of $\mathcal S^{j+1}$ for each $j=0,1,2$. Specifically, letting $\Phi^j$ and $\Psi^j$ be the realization maps for $\mathcal J^j$ and $\mathcal S^j$ respectively, we have that 
 \begin{equation}\label{eq:boundary-vertex-id}
 \Phi^j(w_i)=\begin{cases} \Psi^{j-1}(w_{12-i}) & \text{$i$ even} \\ \Psi^{j+1}(w_i) & \text{$i$ odd} \end{cases} \qquad \Phi^j(y_i)=\begin{cases} \Psi^{j+1}(y_i) & \text{$i$ even} \\ \Psi^{j-1}(y_{12-i}) & \text{$i$ odd} \end{cases}.
 \end{equation}
 In this way, $\Phi^j$ and $\Psi^{j-1}$ exactly agree on the right side of the combinatorial structure $T$ (as shown in Figure \ref{fig:j}), sending the boundary curve to the same $6$-star, while $\Phi^j$ and $\Psi^{j+1}$ map the left side of $T$ to the same regular hexagon with opposite orientations.
 
 We then define $\nabla$ to be the polyhedral surface obtained by gluing the $\mathcal J^j$ and $\mathcal S^j$ together along their shared edges. Formally, we let $\nabla=((V,F),\mathbf{\Phi}^0)$ be the polyhedral surface obtained by setting $V$ to be the disjoint union of the vertex sets of the six tube connections modulo identification of vertices that map to the same point under Equation (\ref{eq:boundary-vertex-id}), and letting $F$ and $\mathbf{\Phi}^0$ be the corresponding unions of faces and vertex realization maps after accounting for the vertex identifications. We refer to the equivalence classes of pairs of vertices identified in this way as the \emph{glued} vertices of $\nabla$. The resulting combinatorial structure of $\nabla$ is shown in Figure \ref{fig:nabla-domain}.
 
  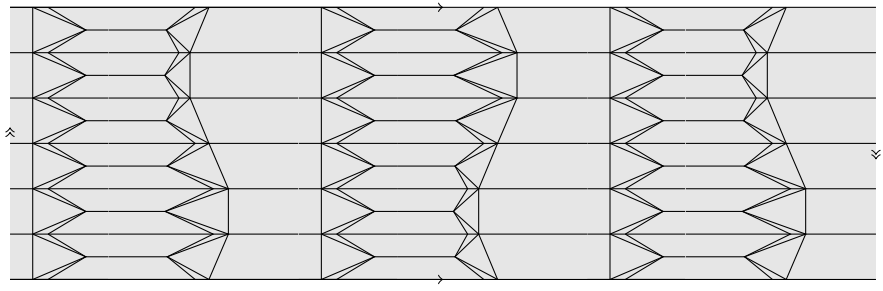
\begin{figure}
  \centering
  
  \begin{tikzpicture}[scale=0.3,rotate=0]
   \newcommand{\strjt}{
    \def\w{4.33983} 
    \def\h{2.33983} 
    \def\l{0.674795} 
    
    \fill[gray!20] (0,0) rectangle (\w,12);

    \draw (1,0) -- (1,12);
    
    \foreach \i in {0,1,...,5}
    {
     \draw (0,{2*\i}) -- (\w,{2*\i});
     \draw (1,{2*\i}) -- ({1+\h},{2*\i+1}) -- (1,{2*\i+2});
     \draw ({1+\h},{2*\i+1}) -- ({\w},{2*\i+1});
     \draw ({1+\l},{2*\i}) -- ({1+\h},{2*\i+1}) -- ({1+\l},{2*\i+2});
    }
   }
   
   \newcommand{\vjt}{
     \def\w{8.38827}
     
     \def\aa{3.52866}
     \def\ab{4.61771}
     \def\ac{3.1}
     \def\ad{4.61771}
     \def\ae{3.52866}
     \def\af{3.82395}
     \def\ag{4.37423}
     \def\ah{3.10927}
     \def\ai{4.79114}
     \def\aj{3.10927}
     \def\ak{4.37423}
     \def\al{3.82395}
     
     \def\ca{2.59407}
     \def\cb{3.1}
     \def\cc{2.5}
     \def\cd{3.1}
     \def\ce{2.59407}
     \def\cf{3.95732}
     \def\cg{2.55177}
     \def\ch{4.79114}
     \def\ci{2.48504}
     \def\cj{4.79114}
     \def\ck{2.55177}
     \def\cl{3.95732}
     
     \fill[gray!20] (0,0) rectangle (\w,12);

     \foreach \i in {2,4,...,10} {\draw (0,\i) -- (\w,\i);}
     \draw ({\w-\ca},1) -- (\w,1);
     \draw ({\w-\cc},3) -- (\w,3);
     \draw ({\w-\ce},5) -- (\w,5);
     \draw ({\w-\cg},7) -- (\w,7);
     \draw ({\w-\ci},9) -- (\w,9);
     \draw ({\w-\ck},11) -- (\w,11);
      
     \draw (\cl,0) -- (\aa,1) -- (\cb,2) -- (\ac,3) -- (\cd,4) -- (\ae,5) -- (\cf,6) -- (\ag,7) -- (\ch,8) -- (\ai,9) -- (\cj,10) -- (\ak,11) -- (\cl,12);
     \draw (\cl,0) -- ({\w-\ca},1) -- (\cb,2) -- ({\w-\cc},3) -- (\cd,4) -- ({\w-\ce},5) -- (\cf,6) -- ({\w-\cg},7) -- (\ch,8) -- ({\w-\ci},9) -- (\cj,10) -- ({\w-\ck},11) -- (\cl,12);
     \draw ({\w-\al},0) -- ({\w-\ca},1) -- ({\w-\ab},2) -- ({\w-\cc},3) -- ({\w-\ad},4) -- ({\w-\ce},5) -- ({\w-\af},6) -- ({\w-\cg},7) -- ({\w-\ah},8) -- ({\w-\ci},9) -- ({\w-\aj},10) -- ({\w-\ck},11) -- ({\w-\al},12);
   }
   
   \def\wj{8.38827}
   \def\ws{4.33983}
   
   \fill[gray!20] (0,0) rectangle ({3*\ws+3*\wj},12);

   \strjt
   
   \begin{scope}[shift={(\ws+\wj,0)},xscale=-1]
    \vjt
   \end{scope}
   
   \begin{scope}[shift={(\ws+\wj,0)}]
    \strjt
   \end{scope}
   
   \begin{scope}[shift={(2*\ws+2*\wj,12)},xscale=-1,yscale=-1]
    \vjt
   \end{scope}
   
   \begin{scope}[shift={(2*\ws+2*\wj,0)}]
    \strjt
   \end{scope}
   
   \begin{scope}[shift={(3*\ws+3*\wj,0)},xscale=-1,]
    \vjt
   \end{scope}

   \draw[->] (0,0) -- ({1.5*\ws+1.5*\wj},0);
   \draw[->] (0,12) -- ({1.5*\ws+1.5*\wj},12);
   
   \draw[->] (0,6.49) -- (0,6.5);
   \draw[->] (0,6.69) -- (0,6.7);
   \draw[->] (3*\ws+3*\wj,5.51) -- (3*\ws+3*\wj,5.5);
   \draw[->] (3*\ws+3*\wj,5.31) -- (3*\ws+3*\wj,5.3);
   
   \draw (0,0) -- ({3*\ws+3*\wj},0);
   \draw (0,12) -- ({3*\ws+3*\wj},12);
   
  
  \end{tikzpicture}
  
  \caption{The combinatorial structure of the polyhedral surface $\nabla$ with the induced length structure after merging coplanar faces.}
  \label{fig:nabla-domain}
 \end{figure}
 
 Theorem A then says that $\nabla$ is flat and immersed, and has the topology of a Klein bottle.
 
 \begin{proof}[Proof of Theorem A]
  We start by proving flatness. The interior vertices of each $\mathcal J^j$ and each $\mathcal S^j$ are flat by Theorems \ref{thm:example-flat} and \ref{thm:straight-flat}, leaving only the glued vertices of $\nabla$. But since each boundary vertex of a tube connection has angle sum $\pi$, and boundary vertices are identified in pairs in the construction of $\nabla$, the angle sum at a glued vertex in $\nabla$ is also $2\pi$.
  
  To prove the immersion property, we start by recalling that all interior vertices of $\mathcal J^j$ and $\mathcal S^j$ are immersed by Theorems \ref{thm:example-locally-injective} and Proposition \ref{prop:str-jt-embedded}. Each glued vertex $u$ is shared by four faces and lies in a plane shared by and separating two neighboring tube connections. The images of the two pairs of faces meeting kitty-corner at $u$ lie on opposite sides of this plane and intersect only at the image of $u$, so by Lemma \ref{lem:kitty}, $\nabla$ is immersed at $u$.
  
  In the construction of $\nabla$, the boundary edges of each tube connection become edges of exactly two faces, so $\nabla$ is a polyhedral surface without boundary, and since vertices of $J^j$ are identified with those of $S^{j-1}$ and $S^{j+1}$, $\nabla$ is connected. Thus since $\nabla$ is a flat connected polyhedral surface without boundary, the discrete Gauss-Bonnet Theorem tells us that $\nabla$ must have the topology of either a torus or Klein bottle. To see that $\nabla$ is not orientable, note that the six copies of the combinatorial structure $T$ are glued along the boundary curves so that the orientation reverses an odd number of times. Formally, let $e$ and $f$ be the upward-oriented left and right boundary curves of $T$ (as shown in Figure \ref{fig:j}), so that $e$ and $f$ belong to the same nontrivial homology class. Then because $\Phi^j$ and $\Psi^{j+1}$ agree along $f$, the vertex identification rule tells us that $f_{\mathcal J^j}=f_{\mathcal S^{j+1}}$, where the subscript indicates whose copy of $T$ the cycle lives in. However, $\Phi^j$ and $\Psi^{j-1}$ reverse the order of vertices for $e$, and so we have $e_{\mathcal J^j}=-e_{\mathcal S^{j-1}}$. Thus we can say
  $$e_{\mathcal J^j}\sim f_{\mathcal J^j}=f_{\mathcal S^{j+1}}\sim e_{\mathcal S^{j+1}}=-e_{\mathcal J^{j+2}},$$
  which, applied three times and recalling indices are taken modulo three, shows $e_{\mathcal J^j}\sim -e_{\mathcal J^{j}}$ and that the homology group has torsion.
 \end{proof}

 To conclude, we briefly mention that coplanar faces meeting at the glued vertices of $\nabla$ can be merged, so these vertices can be omitted from the final vertex count. Each of the six tube connections has 18 interior vertices, giving a total of $108$ vertices. The face count of $168$ can be confirmed in Figure \ref{fig:nabla-domain}.

\section{Acknowledgements}

The author acknowledges support of the Institut Henri Poincar\'e (UAR 839 CNRS-Sorbonne Universit\'e), and LabEx CARMIN (ANR-10-LABX-59-01).

\bibliographystyle{plain}
\bibliography{klein-bib}

\appendix

\section{Appendix: Local Injectivity Proof} \label{sec:appendix}

\begin{proof}[Proof of Theorem \ref{thm:example-locally-injective}]
 As in the theorem, let $\mathscr J=\mathscr J_{6,4,\frac{\pi}{3},\frac{3\pi}{2}}$, and let $\vec\alpha,\vec\gamma$ be the parameter sets generated by $\alpha=3.1,\gamma=2.5$. We will show that $\mathcal C_{\mathscr J,\vec\alpha,\vec\gamma}$ is immersed at each interior vertex $a_i$ and $c_i$ for $i=0,\ldots,6$. By symmetry, the remaining vertices are immersed as well.
 
 For each pair of kitty-corner faces meeting an interior vertex $a_i$ or $c_i$, we find the matrix $M$ from Lemma \ref{lem:m}, and we numerically verify that at least one of the cofactors of $M$ is positive and another is negative. Since the vector of cofactors of a $3\times4$ matrix lies in its kernel, this simultaneously proves that $M$ has rank 3 and that the corresponding condition of Lemma \ref{lem:m} is satisfied.
 
 All calculations were done in Mathematica using interval arithmetic to certify the signs of entries claimed to be nonzero. Specifically for each cofactor calculation, we assigned an interval of radius $10^{-20}$ containing each input parameter, and calculated an interval containing the desired cofactor using outward-rounding interval arithmetic; the sign of the cofactor is certified when this interval does not contain $0$. The radii of the intervals around the cofactors were all less than $10^{-6}$, and in Tables \ref{tab:a-kitty-corner-check} and \ref{tab:c-kitty-corner-check}, we round the center of each interval to 5 decimal places so that the error in these approximations is less than $10^{-5}$.
 
 To organize the data, in Figure \ref{fig:face-names}, we give names to the faces meeting at the $a_i$ and $c_i$ vertices of the domain $T$. For each vertex $a_i$, we define $M_i^{j,k}$ to be the matrix from Lemma \ref{lem:m} determining whether the images of the faces $F_j$ and $F_k$ meeting kitty-corner at $a_i$ as shown in Figure \ref{fig:face-names} intersect away from the image of $a_i$. The cofactors of each such $M_i^{j,k}$ are listed in Table \ref{tab:a-kitty-corner-check}. Similarly for each vertex $c_i$, we define $N_i^{j,k}$ to be the matrix determining whether the images of the faces $G_j$ and $G_k$ meeting kitty-corner at $c_i$ intersect away from the image of $c_i$, and we list the cofactors of each such $N_i^{j,k}$ in Table \ref{tab:c-kitty-corner-check}.
\end{proof}

 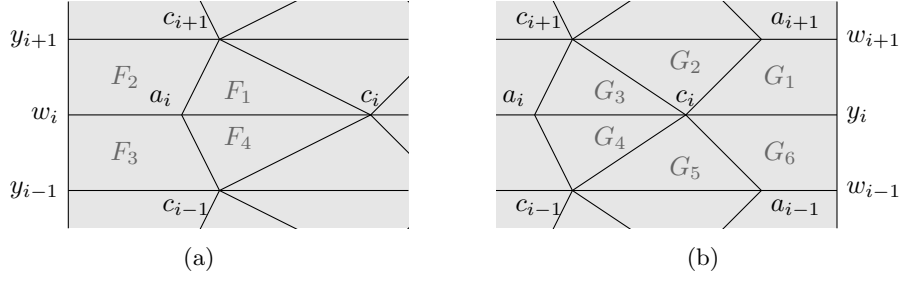
\begin{figure}
  \centering
  \begin{subfigure}[b]{0.45\textwidth}
    \centering
    \begin{tikzpicture}[scale=1]

   \def\w{4.5}
   \def\wa{1.5}
   \def\wb{2}
   \def\wc{4}
   \def\we{5}
   
   \clip (-1,-0.5) rectangle (\w,2.5);
   
   \fill[gray!20] (0,-0.5) -- (\w,-0.5) -- (\w,2.5) -- (0,2.5) -- (0,-0.5);
   
   \foreach \i in {0,...,2}
    {
     \draw (0,\i) -- (\w,\i);
    }
    \draw (0,-0.5) -- (0,2.5);
    
    \foreach \i in {0,2,4}
     {
      \draw (\wb,\i-2) -- (\wa,\i-1) -- (\wb,\i);
      \draw (\wb,\i-2) -- (\wc,\i-1) -- (\wb,\i);
      \draw (\we,\i-2) -- (\wc,\i-1) -- (\we,\i);
     }

   \draw (0,0) node[anchor=east] {$y_{i-1}$};
   \draw (0,1) node[anchor=east] {$w_i$};
   \draw (0,2) node[anchor=east] {$y_{i+1}$};

   \draw (\wa,1) node[anchor=south east] {$a_i$};

   \draw (\wc,1) node[anchor=south] {$c_i$};

   \draw (\wb,0) node[anchor=north east] {$c_{i-1}$};
   \draw (\wb,2) node[anchor=south east] {$c_{i+1}$};
   
   \draw (2.25,1.3) node[black!60] {$F_1$};
   \draw (0.75,1.5) node[black!60] {$F_2$};
   \draw (0.75,0.5) node[black!60] {$F_3$};
   \draw (2.25,0.7) node[black!60] {$F_4$};
   
  \end{tikzpicture}
    \caption{}
    \label{fig:face-names-ai}
  \end{subfigure}
  \hfill 
  \begin{subfigure}[b]{0.45\textwidth}
    \centering
    \begin{tikzpicture}[scale=1]
    \def\w{10}
   \def\h{8}
   \def\wa{6}
   \def\wb{6.5}
   \def\wc{8}
   \def\we{9}
   
   \clip (5.5,-0.5) rectangle (11,2.5);
   
   
   \fill[gray!20] (5,-0.5) -- (\w, -0.5) -- (\w, 2.5) -- (5,2.5) -- (5,-0.5);
   

   \foreach \i in {0,1,2}
    {
     \draw (0,\i) -- (\w,\i);
    }
    \draw (\w,-0.5) -- (\w,2.5);
    
    \foreach \i in {0,2,4}
     {
      \draw (\wb,\i-2) -- (\wa,\i-1) -- (\wb,\i);
      \draw (\wb,\i-2) -- (\wc,\i-1) -- (\wb,\i);
      \draw (\we,\i-2) -- (\wc,\i-1) -- (\we,\i);
     }

   
   \draw (\w,0) node[anchor=west] {$w_{i-1}$};
   \draw (\w,1) node[anchor=west] {$y_i$};
   \draw (\w,2) node[anchor=west] {$w_{i+1}$};

   \draw (\wa,1) node[anchor=south east] {$a_i$};

   \draw (\we,0) node[anchor=north west] {$a_{i-1}$};
   \draw (\we,2) node[anchor=south west] {$a_{i+1}$};
   
   \draw (\wc,1) node[anchor=south] {$c_i$};

   \draw (\wb,0) node[anchor=north east] {$c_{i-1}$};
   \draw (\wb,2) node[anchor=south east] {$c_{i+1}$};
   
   \draw (9.25, 1.5) node[black!60] {$G_1$};
   \draw (\wc, 1.7) node[black!60] {$G_2$};
   \draw (7, 1.3) node[black!60] {$G_3$};
   \draw (7, 0.7) node[black!60] {$G_4$};
   \draw (\wc, 0.3) node[black!60] {$G_5$};
   \draw (9.25, 0.5) node[black!60] {$G_6$};

   \end{tikzpicture}
   \caption{}
    \label{fig:face-names-ci}
  \end{subfigure}
  
  \caption{(a) Here we have zoomed in on a piece of the combinatorial structure $T$ around the vertex $a_i$ and given names to the four faces $F_j$ meeting at $a_i$. (b) Similar for the six faces $G_j$ meeting at $c_i$.}\label{fig:face-names}
 \end{figure}
 
 \begin{table}
  \centering
  \scriptsize
\begin{tabular}{|c|c|c||TTTT|}
 \hline
 $i$ & $j$ & $k$ & \multicolumn{4}{c|}{cofactors of $M_i^{j,k}$}\\
 \hline
 0 & 1 & 3 & \cellcolor{lime}~8.59404 & ~0. & \
\cellcolor{pink}-8.59404 & ~0. \\ 
 0 & 4 & 2 & ~0. & \cellcolor{pink}-8.59404 & ~~0. & \
\cellcolor{lime}8.59404 \\ 
 1 & 1 & 3 & \cellcolor{pink}-0.41733 & \cellcolor{lime}~0.33528 & \
\cellcolor{lime}~0.80361 & \cellcolor{lime}~2.30885 \\ 
 1 & 4 & 2 & \cellcolor{lime}~2.30885 & \cellcolor{lime}~0.41733 & \
\cellcolor{lime}~0.33528 & \cellcolor{pink}-0.80361 \\ 
 2 & 1 & 3 & \cellcolor{lime}~6.66177 & ~0. & \
\cellcolor{pink}-6.66177 & ~0. \\ 
 2 & 4 & 2 & ~0. & \cellcolor{pink}-6.66177 & ~~0. & \
\cellcolor{lime}6.66177 \\ 
 3 & 1 & 3 & \cellcolor{pink}-0.71628 & \cellcolor{lime}~0.3035 & \
\cellcolor{lime}~0.44324 & \cellcolor{lime}~1.91198 \\ 
 3 & 4 & 2 & \cellcolor{lime}~1.91198 & \cellcolor{lime}~0.71628 & \
\cellcolor{lime}~0.3035 & \cellcolor{pink}-0.44324 \\ 
 4 & 1 & 3 & \cellcolor{lime}~2.85207 & ~0. & \
\cellcolor{pink}-2.85207 & ~0. \\ 
 4 & 4 & 2 & ~0. & \cellcolor{pink}-2.85207 & ~~0. & \
\cellcolor{lime}2.85207 \\ 
 5 & 1 & 3 & \cellcolor{pink}-0.61696 & \cellcolor{lime}~0.24393 & \
\cellcolor{lime}~0.07374 & \cellcolor{lime}~1.55463 \\ 
 5 & 4 & 2 & \cellcolor{lime}~1.55463 & \cellcolor{lime}~0.61696 & \
\cellcolor{lime}~0.24393 & \cellcolor{pink}-0.07374 \\ 
 6 & 1 & 3 & \cellcolor{lime}~0.77252 & ~0. & \
\cellcolor{pink}-0.77252 & ~0. \\ 
 6 & 4 & 2 & ~0. & \cellcolor{pink}-0.77252 & ~~0. & \
\cellcolor{lime}0.77252 \\ 
 \hline
\end{tabular}
  \caption{The cofactors of $M_i^{j,k}$ from the proof of Theorem \ref{thm:example-locally-injective}. Values certified positive are highlighted in green, and values certified negative are highlighted in pink.}
    \label{tab:a-kitty-corner-check}
\end{table}

\begin{table}
  
    \centering
    \scriptsize
\begin{tabular}{|c|c|c||TTTT|}
 \hline
 $i$ & $j$ & $k$ & \multicolumn{4}{c|}{cofactors of $N_i^{j,k}$}\\
 \hline    
 0 & 1 & 3 & ~0. & \cellcolor{lime}~0.58525 & \cellcolor{lime}~~1.1705 \
& \cellcolor{pink}-0.58525 \\ 
 0 & 1 & 4 & \cellcolor{pink}-1.1705 & \cellcolor{lime}~0.82767 & \
\cellcolor{lime}~0.82767 & \cellcolor{pink}-1.1705 \\ 
 0 & 1 & 5 & \cellcolor{lime}~3.76147 & \cellcolor{pink}-2.5 & \
\cellcolor{pink}-0.22594 & \cellcolor{pink}-0.82767 \\ 
 0 & 2 & 4 & \cellcolor{pink}-0.58525 & \cellcolor{lime}~1.1705 & \
\cellcolor{lime}~0.58525 & ~0. \\ 
 0 & 2 & 5 & \cellcolor{lime}~1.99371 & \cellcolor{pink}-3.76147 & \
\cellcolor{pink}-2.81953 & \cellcolor{pink}-0.58525 \\ 
 0 & 3 & 5 & \cellcolor{lime}~0.22594 & \cellcolor{pink}-1.99371 & \
\cellcolor{pink}-2.81953 & \cellcolor{pink}-0.58525 \\ 
 0 & 6 & 2 & \cellcolor{pink}-0.58525 & \cellcolor{pink}-2.81953 & \
\cellcolor{pink}-1.99371 & \cellcolor{lime}~0.22594 \\ 
 0 & 6 & 3 & \cellcolor{pink}-0.58525 & \cellcolor{pink}-2.81953 & \
\cellcolor{pink}-3.76147 & \cellcolor{lime}~1.99371 \\ 
 0 & 6 & 4 & \cellcolor{pink}-0.82767 & \cellcolor{pink}-0.22594 & \
\cellcolor{pink}-2.5 & \cellcolor{lime}~3.76147 \\ 
 1 & 1 & 3 & \cellcolor{pink}-0.33528 & \cellcolor{lime}~0.59404 & \
\cellcolor{pink}-0.05119 & \cellcolor{lime}~0.51038 \\ 
 1 & 1 & 4 & \cellcolor{lime}~0.76608 & \cellcolor{pink}-1.44669 & \
\cellcolor{pink}-1.23239 & \cellcolor{lime}~0.05119 \\ 
 1 & 1 & 5 & \cellcolor{pink}-8.74266 & \cellcolor{lime}~2.68468 & \
\cellcolor{lime}~5.8525 & \cellcolor{lime}~1.23239 \\ 
 1 & 2 & 4 & \cellcolor{lime}~0.58525 & \cellcolor{pink}-0.76608 & \
\cellcolor{lime}~0.43379 & \cellcolor{pink}-0.33528 \\ 
 1 & 2 & 5 & \cellcolor{pink}-1.81265 & \cellcolor{lime}~8.74266 & \
\cellcolor{lime}~1.5606 & \cellcolor{pink}-0.43379 \\ 
 1 & 3 & 5 & \cellcolor{pink}-8.59404 & \cellcolor{lime}~1.81265 & \
\cellcolor{pink}-0.70448 & \cellcolor{lime}~0.58525 \\ 
 1 & 6 & 2 & \cellcolor{lime}~0.51038 & \cellcolor{lime}~1.5606 & \
\cellcolor{lime}~1.69265 & \cellcolor{pink}-5.8525 \\ 
 1 & 6 & 3 & \cellcolor{pink}-0.59404 & \cellcolor{pink}-0.70448 & \
\cellcolor{lime}~6.9817 & \cellcolor{pink}-1.69265 \\ 
 1 & 6 & 4 & \cellcolor{lime}~1.44669 & \cellcolor{lime}~8.59404 & \
\cellcolor{lime}~2.68468 & \cellcolor{pink}-6.9817 \\ 
 2 & 1 & 3 & ~0. & \cellcolor{lime}~0.60232 & \
\cellcolor{lime}~1.20464 & \cellcolor{pink}-0.60232 \\ 
 2 & 1 & 4 & \cellcolor{pink}-1.02075 & \cellcolor{lime}~0.75749 & \
\cellcolor{lime}~0.90848 & \cellcolor{pink}-1.20464 \\ 
 2 & 1 & 5 & \cellcolor{lime}~4.03705 & \cellcolor{pink}-2.59407 & \
\cellcolor{pink}-0.4859 & \cellcolor{pink}-0.90848 \\ 
 2 & 2 & 4 & \cellcolor{pink}-0.51038 & \cellcolor{lime}~1.02075 & \
\cellcolor{lime}~0.51038 & ~0. \\ 
 2 & 2 & 5 & \cellcolor{lime}~2.24424 & \cellcolor{pink}-4.03705 & \
\cellcolor{pink}-2.94953 & \cellcolor{pink}-0.51038 \\ 
 2 & 3 & 5 & \cellcolor{lime}~0.45144 & \cellcolor{pink}-2.24424 & \
\cellcolor{pink}-2.94953 & \cellcolor{pink}-0.51038 \\ 
 2 & 6 & 2 & \cellcolor{pink}-0.60232 & \cellcolor{pink}-2.94953 & \
\cellcolor{pink}-2.16539 & \cellcolor{lime}~0.4859 \\ 
 2 & 6 & 3 & \cellcolor{pink}-0.60232 & \cellcolor{pink}-2.94953 & \
\cellcolor{pink}-3.84487 & \cellcolor{lime}~2.16539 \\ 
 2 & 6 & 4 & \cellcolor{pink}-0.75749 & \cellcolor{pink}-0.45144 & \
\cellcolor{pink}-2.59407 & \cellcolor{lime}~3.84487 \\ 
 3 & 1 & 3 & \cellcolor{pink}-0.3035 & \cellcolor{lime}~0.51211 & \
\cellcolor{lime}~0.01981 & \cellcolor{lime}~0.40684 \\ 
 3 & 1 & 4 & \cellcolor{lime}~0.74052 & \cellcolor{pink}-1.21021 & \
\cellcolor{pink}-0.96521 & \cellcolor{pink}-0.01981 \\ 
 3 & 1 & 5 & \cellcolor{pink}-8.05564 & \cellcolor{lime}~3.42714 & \
\cellcolor{lime}~2.58024 & \cellcolor{lime}~0.96521 \\ 
 3 & 2 & 4 & \cellcolor{lime}~0.60232 & \cellcolor{pink}-0.74052 & \
\cellcolor{lime}~0.42056 & \cellcolor{pink}-0.3035 \\ 
 3 & 2 & 5 & \cellcolor{pink}-2.30925 & \cellcolor{lime}~8.05564 & \
\cellcolor{lime}~2.2712 & \cellcolor{pink}-0.42056 \\ 
 3 & 3 & 5 & \cellcolor{pink}-7.47104 & \cellcolor{lime}~2.30925 & \
\cellcolor{pink}-1.5863 & \cellcolor{lime}~0.60232 \\ 
 3 & 6 & 2 & \cellcolor{lime}~0.40684 & \cellcolor{lime}~2.2712 & \
\cellcolor{lime}~1.7059 & \cellcolor{pink}-2.58024 \\ 
 3 & 6 & 3 & \cellcolor{pink}-0.51211 & \cellcolor{pink}-1.5863 & \
\cellcolor{lime}~3.16484 & \cellcolor{pink}-1.7059 \\ 
 3 & 6 & 4 & \cellcolor{lime}~1.21021 & \cellcolor{lime}~7.47104 & \
\cellcolor{lime}~3.42714 & \cellcolor{pink}-3.16484 \\ 
 4 & 1 & 3 & ~0. & \cellcolor{lime}~0.47523 & \
\cellcolor{lime}~0.95047 & \cellcolor{pink}-0.47523 \\ 
 4 & 1 & 4 & \cellcolor{pink}-0.81367 & \cellcolor{lime}~0.64191 & \
\cellcolor{lime}~0.68417 & \cellcolor{pink}-0.95047 \\ 
 4 & 1 & 5 & \cellcolor{lime}~3.5498 & \cellcolor{pink}-2.55177 & \
\cellcolor{pink}-0.50634 & \cellcolor{pink}-0.68417 \\ 
 4 & 2 & 4 & \cellcolor{pink}-0.40684 & \cellcolor{lime}~0.81367 & \
\cellcolor{lime}~0.40684 & ~0. \\ 
 4 & 2 & 5 & \cellcolor{lime}~1.92279 & \cellcolor{pink}-3.5498 & \
\cellcolor{pink}-2.76682 & \cellcolor{pink}-0.40684 \\ 
 4 & 3 & 5 & \cellcolor{lime}~0.29578 & \cellcolor{pink}-1.92279 & \
\cellcolor{pink}-2.76682 & \cellcolor{pink}-0.40684 \\ 
 4 & 6 & 2 & \cellcolor{pink}-0.47523 & \cellcolor{pink}-2.76682 & \
\cellcolor{pink}-2.26319 & \cellcolor{lime}~0.50634 \\ 
 4 & 6 & 3 & \cellcolor{pink}-0.47523 & \cellcolor{pink}-2.76682 & \
\cellcolor{pink}-4.02004 & \cellcolor{lime}~2.26319 \\ 
 4 & 6 & 4 & \cellcolor{pink}-0.64191 & \cellcolor{pink}-0.29578 & \
\cellcolor{pink}-2.55177 & \cellcolor{lime}~4.02004 \\ 
 5 & 1 & 3 & \cellcolor{pink}-0.24393 & \cellcolor{lime}~0.3884 & \
\cellcolor{pink}-0.17522 & \cellcolor{lime}~0.35487 \\ 
 5 & 1 & 4 & \cellcolor{lime}~0.39661 & \cellcolor{pink}-0.97287 & \
\cellcolor{pink}-0.8857 & \cellcolor{lime}~0.17522 \\ 
 5 & 1 & 5 & \cellcolor{pink}-4.53551 & \cellcolor{lime}~4.14925 & \
\cellcolor{lime}~0.77252 & \cellcolor{lime}~0.8857 \\ 
 5 & 2 & 4 & \cellcolor{lime}~0.47523 & \cellcolor{pink}-0.39661 & \
\cellcolor{lime}~0.42976 & \cellcolor{pink}-0.24393 \\ 
 5 & 2 & 5 & \cellcolor{pink}-2.04962 & \cellcolor{lime}~4.53551 & \
\cellcolor{lime}~1.44239 & \cellcolor{pink}-0.42976 \\ 
 5 & 3 & 5 & \cellcolor{pink}-3.1239 & \cellcolor{lime}~2.04962 & \
\cellcolor{pink}-0.43166 & \cellcolor{lime}~0.47523 \\ 
 5 & 6 & 2 & \cellcolor{lime}~0.35487 & \cellcolor{lime}~1.44239 & \
\cellcolor{lime}~1.66866 & \cellcolor{pink}-0.77252 \\ 
 5 & 6 & 3 & \cellcolor{pink}-0.3884 & \cellcolor{pink}-0.43166 & \
\cellcolor{lime}~1.66941 & \cellcolor{pink}-1.66866 \\ 
 5 & 6 & 4 & \cellcolor{lime}~0.97287 & \cellcolor{lime}~3.1239 & \
\cellcolor{lime}~4.14925 & \cellcolor{pink}-1.66941 \\ 
 6 & 1 & 3 & ~0. & \cellcolor{lime}~0.35487 & \
\cellcolor{lime}~0.70974 & \cellcolor{pink}-0.35487 \\ 
 6 & 1 & 4 & \cellcolor{pink}-0.70974 & \cellcolor{lime}~0.50186 & \
\cellcolor{lime}~0.50186 & \cellcolor{pink}-0.70974 \\ 
 6 & 1 & 5 & \cellcolor{lime}~3.57331 & \cellcolor{pink}-2.48504 & \
\cellcolor{pink}-0.05893 & \cellcolor{pink}-0.50186 \\ 
 6 & 2 & 4 & \cellcolor{pink}-0.35487 & \cellcolor{lime}~0.70974 & \
\cellcolor{lime}~0.35487 & ~0. \\ 
 6 & 2 & 5 & \cellcolor{lime}~1.81612 & \cellcolor{pink}-3.57331 & \
\cellcolor{pink}-2.56838 & \cellcolor{pink}-0.35487 \\ 
 6 & 3 & 5 & \cellcolor{lime}~0.05893 & \cellcolor{pink}-1.81612 & \
\cellcolor{pink}-2.56838 & \cellcolor{pink}-0.35487 \\ 
 6 & 6 & 2 & \cellcolor{pink}-0.35487 & \cellcolor{pink}-2.56838 & \
\cellcolor{pink}-1.81612 & \cellcolor{lime}~0.05893 \\ 
 6 & 6 & 3 & \cellcolor{pink}-0.35487 & \cellcolor{pink}-2.56838 & \
\cellcolor{pink}-3.57331 & \cellcolor{lime}~1.81612 \\ 
 6 & 6 & 4 & \cellcolor{pink}-0.50186 & \cellcolor{pink}-0.05893 & \
\cellcolor{pink}-2.48504 & \cellcolor{lime}~3.57331 \\ 
\hline
\end{tabular}

  \caption{The cofactors of $N_i^{j,k}$ from the proof of Theorem \ref{thm:example-locally-injective}. Values certified positive are highlighted in green, and values certified negative are highlighted in pink.}
  \label{tab:c-kitty-corner-check}
 \end{table}

\end{document}